\documentclass[reqno]{amsart}
\usepackage{color}
\usepackage{amssymb,amsmath,amsthm,amstext,amsfonts,amscd}
\usepackage[dvips]{graphicx}
\usepackage{psfrag}
\usepackage{url}

\makeatletter \@addtoreset{equation}{section} \makeatother

\renewcommand\thetable{\thesection.\@arabic\c@table}

\theoremstyle{plain}

\newtheorem{theorem}{Theorem}[section]

\newtheorem{lemma}{Lemma}[section]
\newtheorem{corollary}{Corollary}[section]
\newtheorem{definition}{Definition}[section]
\newtheorem{remark}{Remark}[section]

\newtheorem{claim}{Claim}[section]
\newtheorem{assumption}{Assumption}[section]

\newcommand{\supp}{\operatorname{supp}}
\newcommand{\diam}{\operatorname{diam}}

\begin{document}

\title{The approximation of Lyapunov exponents by horseshoes for $C^1$-diffeomorphisms with dominated splitting}

\author{Juan Wang}
\address{School of mathematics, physics and statistics, Shanghai University of Engineering Science, Shanghai 201620, P.R. China}
\email{wangjuanmath@sues.edu.cn}

\author{Rui Zou}
\address{School of Mathematics and Statistics, Nanjing University of Information Science and Technology, Nanjing 210044,  P.R. China}
\email{mathzourui@gmail.com }

\author{Yongluo Cao}
\address{Departament of Mathematics, Soochow University,
Suzhou 215006, Jiangsu, P.R. China}
\address{Departament of Mathematics, Shanghai Key Laboratory of PMMP, East China Normal University,
 Shanghai 200062, P.R. China}
\email{ylcao@suda.edu.cn}

\thanks{The first author is partially supported by NSFC (11501400, 11871361) and the Talent Program of Shanghai University of Engineering Science. The third author is partially supported by NSFC (11771317, 11790274).}

\date{\today}

\begin{abstract}
Let $f$ be a $C^1$-diffeomorphism and $\mu$ be  a hyperbolic ergodic $f$-invariant Borel probability measure with positive measure-theoretic entropy.
 Assume that the  Oseledec splitting
$$T_xM=E_1(x) \oplus\cdots\oplus E_s(x) \oplus E_{s+1}(x) \oplus\cdots\oplus E_l(x) $$
is dominated on the Oseledec basin $\Gamma$.  We give extensions of Katok's Horseshoes construction. Moreover there is  a dominated splitting corresponding to Oseledec subspace on horseshoes.
\end{abstract}


\footnotetext{2010 {\it Mathematics Subject classification}:
 }

\maketitle


\section{Introduction}
Let $f$ be a $C^r$ $(r\geq1)$ diffeomorphism of a compact Riemannian manifold $M$. An $f$-invariant subset $\Lambda\subset M$ is called a  \emph{hyperbolic set} if there exists a continuous splitting of the tangent bundle $T_\Lambda M = E^{s}\oplus E^{u}$, and constants $c > 0,\ 0 < \tau < 1$ such that for every $x \in \Lambda$,
 \begin{enumerate}
\item[(1)] $d_xf(E^s(x)) = E^s(f(x)),\ d_xf(E^u(x)) = E^u(f(x))$;
\item[(2)] for all $n \geq 0,\ \|d_xf^n(v)\|\leq c\tau ^n\|v\|$ if $v \in E^s(x)$, and $\|d_xf^{-n}(v)\|\leq c\tau^n\|v\|$ if $v \in E^u(x)$.
\end{enumerate}
A hyperbolic set $\Lambda$ is called \emph{locally maximal}, if there exists a neighbourhood $U$ of $\Lambda$
such that $\Lambda=\bigcap_{n\in\mathbb{Z}}f^n(U)$. Let $\mathcal{M}_f(\Lambda)$ be the space of all $f$-invariant Borel probability measures on $\Lambda$.
Let $\mu$ be a hyperbolic ergodic $f$-invariant Borel probability measure on $M$. We say $\mu$ \emph{hyperbolic} if it possesses at least one negative and one positive, and no zero Lyapunov exponents. Let $\Gamma$ be the Oseledec's basin of $\mu$ (see Theorem \ref{oseledectheo}). For $x\in\Gamma$, denote its distinct Lyapunov exponents by
$$\lambda_1(\mu)<\cdots<\lambda_s(\mu)<0<\lambda_{s+1}(\mu)<\cdots<\lambda_l(\mu)$$
with multiplicities $n_1, n_2, \cdots, n_l\geq1$ and let
$$T_xM=E_1(x)\oplus\cdots\oplus E_s(x)\oplus E_{s+1}(x)\oplus\cdots\oplus E_l(x)$$
be the corresponding decomposition of its tangent space. Denote $E^s=E_1\oplus\cdots\oplus E_s$ and $E^u=E_{s+1}\oplus\cdots\oplus E_l$.
In this paper, we consider $r=1$ and the following assumption.
\begin{assumption}\label{assumpdominated}
The splitting
$$T_{\Gamma} M=E_1\oplus\cdots\oplus E_s\oplus E_{s+1}\oplus\cdots\oplus E_l$$
is dominated.
\end{assumption}
We state the main result of this paper:

\begin{theorem}\label{maintheorem}
Let $f: M\rightarrow M$ be a $C^1$ diffeomorphism of a compact Riemannian manifold $M$ and $\mu$ be a hyperbolic ergodic $f$-invariant Borel probability measure on $M$ with positive measure-theoretic entropy $h_\mu(f)>0$.
Under Assumption \ref{assumpdominated},
we have for every small $\varepsilon>0$, there exists a compact set $\Lambda^*\subseteq M$ and a positive integer $m$ satisfying
\begin{itemize}
\item[(i)] $\Lambda^*$ is a locally maximal hyperbolic set and topologically mixing with respect to $f^m$.
\item[(ii)] $|h_{top}(f|_\Lambda)-h_\mu(f)|<\varepsilon$ where $\Lambda=\Lambda^*\cup f(\Lambda^*)\cup\cdots\cup f^{m-1}(\Lambda^*)$.
\item[(iii)] $\Lambda$ is contained in the $\varepsilon$-neighborhood of the support of $\mu$.
\item[(iv)] $d(\nu,\mu)<\varepsilon$ for every $\nu\in\mathcal{M}_f(\Lambda)$, where $d$ is a metric that generates the weak* topology.
\item[(v)] There is a dominated splitting $T_{\Lambda} M=\widetilde{E}_1\oplus_< \widetilde{E}_2\oplus_< \cdots\oplus_< \widetilde{E}_l$ on $\Lambda$ with $\dim \widetilde{E}_i=n_i$, and
\begin{equation}\label{lyexloup}
e^{[\lambda_{i}(\mu)-6\varepsilon]km}\|u\|\leq \|d_xf^{km}(u)\|\leq e^{[\lambda_{i}(\mu)+6\varepsilon]km}\|u\|
\end{equation}
for every $x\in\Lambda$, $k\geq 1$ and $0\neq u\in \widetilde{E}_i(x)$, $i=1, \cdots, l$.
\end{itemize}
\end{theorem}

This paper is motived by Katok \cite{Ka80}, Avila, Crovisier, Wilkinson \cite{ACW17}  and Cao, Pesin, Zhao \cite{cpz17}.
Katok \cite{Ka80} proved for a $C^2$ diffeomorphism $f$ preserving an ergodic hyperbolic measure with positive entropy, there exists  a sequence of horseshoes, and the topology entropy of $f$ restricted to horseshoes can be arbitrarily close to the measure-theoretic entropy, implying an abundance of hyperbolic periodic points.
Mendoza \cite{Me88}  proved that an ergodic hyperbolic SRB measure of $C^2$ surface diffeomorphism can be approximated by a sequence of measures supported on horseshoes and the Hausdorff dimension for horseshoes on the unstable manifold converges to $1$.
 Katok and Mendoza also elaborated on the related results in the part of supplement of \cite{KH95}.  Avila, Crovisier, Wilkinson \cite{ACW17} explicitly gave a dominated  splitting $T_\Lambda M=E_1\oplus_<\cdots\oplus_<E_l$ on each horseshoe $\Lambda$ and the Lyapunov exponential approximation in each subbundle $E_i$ is obtained over $\Lambda$, for $i=1,2,\cdots,l$.   For $C^r (r>1)$ maps, results related to Katok's approximation were obtained by Chung \cite{Chung}, Gelfert \cite{Gelfert10} and Yang \cite{Yang}.  For every ergodic invariant measure $\mu$  with positive  entropy for  $C^{1+\alpha}$ nonconformal repellers,  Cao, Pesin, Zhao \cite{cpz17} constructed a compact  expanding invariant set with  dominated splitting corresponding to Oseledec splitting of $\mu$, and for which  entropy and Lyapunov exponents approximate to entropy and Lyapunov exponents for $\mu$. Then  Cao, Pesin, Zhao\cite{cpz17}  used this construction  to show the continuity of sub-additive topological pressure and give a sharp estimate for the lower bound estimate of Hausdorff dimension of non-conformal repellers.
Lian and Young extended the results of Katok \cite{Ka80} to mappings of Hilbert spaces \cite{ly11} and to semi-flows on Hilbert spaces \cite{ly12}.    In the case of  Banach quasi-compact cocycles for $C^r (r>1)$ diffeomorphism, for an  ergodic hyperbolic measure with positive entropy, Zou and Cao \cite{zc17}  constructed a Horseshoe with dominated splitting corresponding to Oseledec splitting of $\mu$   for cocycles, for which there are entropy and Lyapunov exponents's  approximations.

Gelfert \cite{Gelfert16} relaxed the smoothness to $C^1$. Gelfert's results assert the following: let $f$ be a $C^1$ diffeomorphism of a smooth Riemannian manifold, and let $\mu$ be an ergodic hyperbolic $f$-invariant Borel probability measure whose support admits a dominated splitting $T_{\supp\mu}M=E^s\oplus_<E^u$. Assume that $(f,\mu)$ has positive measure-theoretic entropy. Then she proved the analogous results of Katok \cite{Ka80}. In this paper we assume the splitting
$$T_{\Gamma} M=E_1\oplus\cdots\oplus E_s\oplus E_{s+1}\oplus\cdots\oplus E_l$$
is dominated. We also prove the existence of locally maximal hyperbolic sets (horseshoes) in a neighborhood of the support set of $\mu$  by using Katok's technique. Moreover we use some properties of the dominated splitting to show the invariance of the cones. Then we obtain a dominated splitting $T_\Lambda M=\widetilde{E}_1\oplus_<\cdots\oplus_<\widetilde{E}_l$ on each horseshoe $\Lambda$.
We also use the properties of $C^1$ nonuniform hyperbolic dynamical systems to prove the approximation of Lyapunov exponents on each subbundle $\widetilde{E}_i$ over $\Lambda$ (see (\ref{lyexloup})). We didn't use the Lyapunov charts and Lyapunov norm in the proof.

The remainder of this paper is organized as follows. In Section 2, we recall  some related notations, properties and theorems. Section 3 provides the proof of the main result.

\vskip1cm


\section{Preliminaries}
Let $M$ be a compact Riemannion manifold, and $f$
be a $C^1$ diffeomorphism from $M$ to itself. Let $\dim M$ be the dimension of M. For $x\in M$, we denote
$$\| d_xf\|= \sup_{0\neq u \in T_xM}\displaystyle
\frac{\|d_xf(u)\|}{\|u\|},\ \ m(d_xf) = \inf_{0 \neq u \in
T_xM}\displaystyle \frac{\|d_xf(u)\|}{\|u\|}$$
which are respectively called the maximal norm and minimum norm of the differentiable operator $d_xf : T_xM \to T_{fx}M$, where $\|\cdot\|$ is the norm induced by the Riemannian metric $d$ on $M$.
Let $n$ be a natural number we define a metric $d_n$ on $M$ by $d_n(x,y)=\max_{0\leq i\leq n-1}d(f^ix,f^iy)$.
For any $\rho>0$ and $x\in M$, \emph{an $(n,\rho)$ Bowen ball of $x$} is
$$B_n(x,\rho)=\Big\{y\in M: d_n(x,y) < \rho\Big\}.$$
A subset $E$ of $M$ is said to be \emph{ $(n,\rho)$-separated with respect to $f$ } if $x,y\in E$, $x\neq y$, implies $d_n(x,y)>\rho$.

\subsection{Dominated splitting}
Ma$\tilde{\text{n}}\acute{\text{e}}$, Liao and Pliss introduced independently the concept of dominated splitting in order to prove that structurally stable systems satisfy a hyperbolic condition on the tangent map. We recall the definition and some properties of the dominated splitting(see Appendix B in \cite{BDV}).
\begin{definition}\label{dominated-splitting}
Let $K \subseteq M$ be an $f$-invariant subset. A
$df$-invariant splitting $T_K M = E_1 \oplus E_2 \oplus \cdots \oplus E_k$ of the tangent
bundle over $K$ is \emph{dominated} if there exists $N\in \mathbb{N}$
such that for every $i<j$, every $x\in K$, and each pair of unit vectors $u\in E_i(x)$
and $v\in E_j(x)$, one has
$$\frac{\|d_xf^N(u)\|}{\|d_xf^N(v)\|}\leq\frac{1}{2}$$
and the dimension of $E_i(x)$ is independent of $x\in K$ for every $i\in\{1,\cdots,k\}$. We denote $T_K M = E_1 \oplus_< E_2\oplus_<\cdots\oplus_<E_k$.
\end{definition}

\begin{remark}\label{conds}
\begin{itemize}
\item [ (1) ] We can assume $N=1$ for a smooth change of metric on $M$.
\item [ (2) ] The dominated splitting $T_K M = E_1 \oplus_< E_2\oplus_<\cdots\oplus_<E_k$ can be extended to a dominated splitting $TM = E_1 \oplus_< E_2\oplus_<\cdots\oplus_<E_k$ over the closure of $K$. See \cite{BDV} for a proof.
\item [ (3) ] The dominated splitting $T_K M = E_1 \oplus_< E_2\oplus_<\cdots\oplus_<E_k$ can be extended to a continuous splitting $TM = E_1 \oplus E_2\oplus\cdots\oplus E_k$ in a neighborhood of $K$. See \cite{BDV} for a proof.
\item [ (4) ] The dominated splitting is unique if one fixes the dimensions of the subbundles.
\item [ (5) ] Every dominated splitting is continuous, i.e. the subspaces $E_i(x)$, $i=1$, $2,\cdots, k$, depend continuously on the point $x$.
\end{itemize}
\end{remark}

\subsection{Preliminaries of Lyapunov exponents}
We review the Oseledec's Theorem which contains the definitions of the Oseledec's basin, the Oseledec's splitting, the Lyapunov exponents and the multiplicities.
\begin{theorem}\label{oseledectheo}(\cite{Ose68})
Let $f$ be a $C^1$ diffeomorphism of a compact Riemannian manifold $M$ preserving an ergodic $f$-invariant Borel probability measure $\mu$. Then there exist
\begin{itemize}
\item [(1)] real numbers $\lambda_1(\mu)<\lambda_2(\mu)<\cdots<\lambda_l(\mu)(l\leq \dim M)$;
\item [(2)] positive intergers $n_1, n_2, \cdots, n_l$, satisfying $n_1+\cdots+n_l=\dim M$;
\item [(3)] a Borel set $\Gamma:=\Gamma(\mu)$, called the Oseledec's basin of $\mu$, satisfying $f\big(\Gamma)=\Gamma$ and $\mu\big(\Gamma\big)=1$;
\item [(4)] a measurable splitting, called the Oseledec's splitting, $T_xM=E_1(x)\oplus\cdots \oplus E_l(x)$ with $\dim E_i(x)=n_i$ and $d_xf(E_i(x))=E_i(f(x))$, such that
$$\lim_{n\to\pm\infty}\frac{\log\|d_xf^n(v)\|}{n}=\lambda_i(\mu),$$
for any $x\in\Gamma$, $0\not=v\in E_i(x)$, $i=1,2,\cdots,l$.
\end{itemize}
The real numbers $\lambda_1(\mu),\lambda_2(\mu),\cdots,\lambda_l(\mu)$ are called the \emph{Lyapunov exponents}, and $n_1,n_2,\cdots,n_l$ are called the \emph{multiplicities}.
\end{theorem}

In the result to follow, we will review the Lyapunov exponents as a limit of Birkhoff sums in terms of $d_{(\cdot)} f^N$ for natural numbers $N$ large enough, which is computed in \cite{ABC11}.
\begin{lemma}\label{malyexoffN}(Lemma $8.4$ in \cite{ABC11})
Let $f$ be a $C^1$-diffeomorphism, $\mu$ be an ergodic invariant
probability measure, and $E\subseteq T_{\text{supp}(\mu)}M$ be a
$df$-invariant continuous subbundle defined over the support of $\mu$. Let $\lambda_E^+$ be the upper Lyapunov
exponent in $E$ of the measure $\mu$.
Then, for any $\varepsilon>0$, there exists an integer
$N_1(\varepsilon)$ such that, for $\mu$ almost every point $x\in M$
and any $N\geq N_1(\varepsilon)$, the Birkhoff averages
$$\frac{1}{kN}\sum_{l=0}^{k-1}\log\|d_{f^{lN}(x))}f^N|_{E(f^{lN}(x))}\|$$
converge towards a number contained in
$[\lambda_E^+,\lambda_E^++\varepsilon)$, where $k$ goes to
$+\infty$.
\end{lemma}

The following Lemma is analogous.
\begin{lemma}\label{milyexoffN}
Let $f$ be a $C^1$-diffeomorphism, $\mu$ be an ergodic invariant
probability measure, and $E\subseteq T_{\text{supp}(\mu)}M$ be a
$df$-invariant continuous subbundle defined over $\text{supp}(\mu)$. Let
$\lambda_E^-$ be the lower Lyapunov exponent in $E$ of the
measure $\mu$.
Then, for any $\varepsilon>0$, there exists an integer
$N_2(\varepsilon)$ such that, for $\mu$ almost every point $x\in M$
and any $N\geq N_2(\varepsilon)$, the Birkhoff averages
$$\frac{1}{kN}\sum_{l=0}^{k-1}\log m(d_{f^{lN}(x)}f^N|_{E(f^{lN}(x))})$$
converge towards a number contained in
$(\lambda_E^--\varepsilon,\lambda_E^-]$, where $k$ goes to
$+\infty$.
\end{lemma}
\begin{proof}
It is a slight modification of the proof of Lemma \ref{malyexoffN}. We
omit it here.
\end{proof}

\subsection{ $(\rho,\beta,\gamma)$-rectangle of a compact subset}
Let $(f,M)$ be as above and $\mu$ be a hyperbolic ergodic $f$-invariant Borel probability measure on $M$.
For $\mu$ almost every $x\in M$, denote the Lyapunov exponents by
$$\lambda_1(\mu)<\cdots<\lambda_s(\mu)<0<\lambda_{s+1}(\mu)<\cdots<\lambda_l(\mu),$$
and the corresponding decomposition of its tangent space by $T_xM=E_1(x)\oplus \cdots \oplus E_l(x)$.
Denote $E^{s}=E_1\oplus\cdots\oplus E_s$ and $E^{u}=E_{s+1}\oplus\cdots\oplus E_l$.
Let $d_s=\dim E^s$, $d_u=\dim E^u$ with $d_s+d_u=\dim M$. Let $I=[-1,1]$, given $x\in M$, we say $R(x)\subseteq M$ is a \emph{rectangle} in $M$ centered at $x$ if there exists a $C^1$-embedding $\Phi_x:I^{\dim M}\to M$ such that $\Phi_x(I^{\dim M})=R(x)$ and $\Phi_x(0)=x$. A set $\widetilde{H}$ is called an \emph{admissible $u$-rectangle} in $R(x)$, if there exist $0<\widetilde{\lambda}<1$, $C^1$-maps $\phi_1,\phi_2: I^{d_u}\to I^{d_s}$ satisfying $\|\phi_1(u)\|\geq\|\phi_2(u)\|$ for $u\in I^{d_u}$ and $\|d\phi_i\|\leq\widetilde{\lambda}$ for $i=1,2$ such that $\widetilde{H}=\Phi_x(H)$, where
$$H=\{(u,v)\in I^{d_u}\times I^{d_s}:v=t\phi_1(u)+(1-t)\phi_2(u),0\leq t\leq1\}.$$
Similarly we can define an \emph{admissible $s$-rectangle} in $R(x)$.

\begin{definition}\label{smallrectangle}
Let $(f,M)$ be as above and $\Lambda\subseteq M$ be compact. We say $R(x)$ is a $(\rho,\beta,\gamma)$-rectangle of $\Lambda$ for $\rho>\beta>0,\gamma>0$, if there exists $\widetilde{\lambda}=\widetilde{\lambda}(\rho,\beta,\gamma)$ satisfying
\begin{itemize}
\item[(1)] $x\in\Lambda$, $B(x,\beta)\subseteq int\ R(x)$ and $\diam R(x)\leq\frac{\rho}{3}$.
\item[(2)] If $z,f^mz\in\Lambda\cap B(x,\beta)$ for some $m>0$, then the connected component $C(z,R(x)\cap f^{-m}R(x))$ of $R(x)\cap f^{-m}R(x)$ containing $z$ is an admissible $s$-rectangle in $R(x)$, and $f^m(C(z,R(x)\cap f^{-m}R(x)))$ is an admissible $u$-rectangle in $R(x)$.
\item[(3)] $\diam f^k(C(z,R(x)\cap f^{-m}R(x)))\leq\rho e^{-\gamma\min\{k,m-k\}}$ for $0\leq k\leq m$.
\end{itemize}
\end{definition}

Gelfert \cite{Gelfert16} proved that there is a finite collection of $(\rho,\beta,\gamma)$-rectangles for $(f,M,\mu)$. We only state the lemma here. See Lemma $2$ in \cite{Gelfert16} for a proof. We also refer to \cite{KH95} for more information about $(\rho,\beta,\gamma)$-rectangles.
\begin{lemma}\label{existofrectangle}
Let $f$ be a $C^1$-diffeomorphsim of a compact Riemannian manifold $M$ and $\mu$ be a hyperbolic ergodic $f$-invariant Borel probability measure on $M$.
Assume that $T_{\supp\mu}M=E^{s}\oplus_< E^{u}$ is dominated. Let $\chi:=\lambda(\mu)=\min_{1\leq i\leq l}|\lambda_i(\mu)|$. Given $\rho>0$ and $\delta>0$,
then there exists a compact set $\Lambda_H=\Lambda_H(\rho,\delta,\frac{\chi}{2})$ with $\mu(\Lambda_H)>1-\frac{\delta}{3}$, a constant $\beta=\beta(\rho,\delta)>0$ and a finite collection of $(\rho,\beta,\frac{\chi}{2})$-rectangles $R(q_1),R(q_2),\cdots,R(q_t)$ with $q_j\in\Lambda_H$ so that $\Lambda_H\subseteq\bigcup_{j=1}^tB(q_j,\beta)$.
\end{lemma}

\subsection{A necessary and sufficient condition for hyperbolic sets}
Let $T_xM=F(x)\oplus G(x)$ be a splitting of the tangent space at $x\in M$ and let $\theta\in(0,1)$ be small, we define the cones at $x$ with respect to this decomposition of $TM$  by
$$C^F_\theta(x)=\big{\{}v+w \in F(x)\oplus G(x): v\in F(x), w\in G(x)\ \text{and}\ \|w\|\leq\theta\|v\|\big{\}},$$
and $$C^G_\theta(x)=\big{\{}v+w\in F(x)\oplus G(x): v\in F(x), w\in G(x)\ \text{and}\ \|v\|\leq\theta\|w\|\big{\}}.$$
We define a projection $\pi_F(x)$ from $T_xM$ to $F(x)$ with the kernel $G(x)$ by
$\pi_F(x)(u)=v$, for every $u=v+w\in T_xM$ with $v\in F(x)$ and $w\in G(x)$. Similarly we can define $\pi_G(x)$.
The following Theorem is Theorem $6.1.2$ in \cite{BP07}, which gives a characterization of hyperbolic sets in terms of cones.

\begin{theorem}\label{hyperbolicity}
A compact invariant set $\Lambda\subseteq M$ of a $C^1$ diffeomorphism $f$ of a compact manifold is hyperbolic if and only if there exists a Riemannian metric on $M$, a continuous splitting
$$T_xM=F(x)\oplus G(x)\ \text{for each } x\in\Lambda,$$
a constant $\lambda\in(0,1)$, and a continuous function $\theta: \Lambda\to\mathbb{R}^+$ such that for $x\in\Lambda$,
\begin{itemize}
\item [ (1) ] $d_xf \Big{(} C^G_{\theta(x)}(x)\Big{)}\subseteq C^G_{\theta(fx)}(fx)$ and $d_xf^{-1} \Big{(} C^F_{\theta(x)}(x)\Big)\subseteq C^F_{\theta(f^{-1}x)}(f^{-1}x)$,
\item [ (2) ] If $v\in C^G_{\theta(x)}(x)$, then $\|d_xf(v)\|\geq\lambda^{-1}\|v\|$. If $v\in C^F_{\theta(x)}(x)$, then $\|d_xf^{-1}(v)\|\geq\lambda^{-1}\|v\|$.
\end{itemize}
\end{theorem}




\section{The proof of the main theorem}
This section provides the proof of the main result stated in Section $1$. First of all, we construct the subset $\Lambda^*$ of $M$ by using Katok's technique \cite{Ka80} and $\Lambda^*$ is $f^m$-invariant for some sufficiently large positive integer $m$. Denote $\Lambda=\Lambda^*\cup f(\Lambda^*)\cup\cdots\cup f^{m-1}(\Lambda^*)$. We give that $\Lambda$ is in a neighborhood of $\supp\mu$. Then we use some properties of the dominated splitting and $C^1$ nonuniform hyperbolic dynamical systems to prove Lemma \ref{invariantofcone} and Corollary \ref{corofle}. We also obtain a dominated splitting $T_\Lambda M=\widetilde{F}_j\oplus\widetilde{G}_j$ over $\Lambda$ for each $j=1,2,\cdots,l-1$ in the proof of Lemma \ref{invariantofcone}. Therefore Lemma \ref{invariantofcone} and Theorem \ref{hyperbolicity} tell us that $\Lambda^*$ is a hyperbolic set with respect to $f^m$. Finally we prove there is a dominated splitting $T_\Lambda M=\widetilde{E}_1\oplus\widetilde{E}_2\oplus\cdots\oplus\widetilde{E}_{l}$ corresponding to Oseledec subspace over $\Lambda$, where $\widetilde{E}_1=\widetilde{F}_1$, $\widetilde{E}_j=\widetilde{F}_j\cap\widetilde{G}_{j-1}$ for $j=2,3,\cdots,l-1$ and $\widetilde{E}_l=\widetilde{G}_{l-1}$.
\begin{proof}
Since the splitting
$$T_{\Gamma}M=E_1\oplus\cdots\oplus E_l$$
is dominated splitting and $\supp(\mu)\subseteq\overline{\Gamma}$ (see P$693$ in \cite{KH95}), by Remark \ref{conds}, we have
$T_{\supp(\mu)}M=E_1\oplus \cdots \oplus E_s\oplus E_{s+1}\oplus \cdots \oplus E_l$ is a dominated splitting. Therefore the angles between $E_i(x)$ and $E_j(x)$ are uniformly bounded from zero for every $i,j \in\{1,2,\cdots, l\}$ with $i\neq j$ and $x \in \supp(\mu)$. Combing with Remark \ref{conds}, there is a small $0<\varepsilon_0<1$ satisfying the two properties:
\begin{itemize}
\item [ (1) ] we can extend the dominated splitting $$T_{\supp(\mu)}M=E_1\oplus_<\cdots\oplus_<E_s\oplus_<E_{s+1}\oplus_<\cdots\oplus_<E_l$$ to a continuous splitting in the $\varepsilon_0$-neighborhood $\mathcal{U}(\varepsilon_0, \supp(\mu))$ of $\supp(\mu)$.
\item [ (2) ] for any $x,y\in M$ with $d(x,y)<\varepsilon_0$, there exists a unique geodesic connecting $x$ and $y$.
\end{itemize}
For every $j\in\{1, \cdots, l-1\}$ we denote
$$F_j=E_1\oplus\cdots\oplus E_j \text{ and } G_j=E_{j+1}\oplus\cdots\oplus E_l.$$
It is easy to see $T_xM=F_j(x)\oplus G_j(x)$ for any $x\in\mathcal{U}(\varepsilon_0,\supp(\mu))$.
For any small $\theta\in(0,1)$, $x,y\in\mathcal{U}(\varepsilon_0,\supp(\mu))$ with $d(x,y)<\varepsilon_0$ and every $u\in C^{G_j}_\theta(y)$, let
$\widetilde{u}$ be the parallel transport of $u$ from $T_yM$ to $T_xM$ along the geodesic connecting $y$ and $x$.
By the Whitney Embedding Theorem, we assume, without loss of generality, that the manifold $M$ is embedded in $\mathbb{R}^{N_3}$ with a sufficiently large $N_3$. Since $f$ is a $C^1$ diffeomorphism and the splitting $T_{\mathcal{U}(\varepsilon_0, \supp(\mu))}M=F_j\oplus G_j$ is continuous, for the above $\varepsilon_0>0$, there is a $\rho_0\in(0,\frac{1}{2}\varepsilon_0)$ such that
\begin{itemize}
\item [ (3) ] for any $x,y\in\mathcal{U}(\varepsilon_0,\supp(\mu))$ with $d(x,y)<\rho_0$, if $u\in C^{G_j}_\theta(y)$, then $\widetilde{u}\in C^{G_j}_{\frac32\theta}(x)$.
\item [ (4) ] let $V=\mathcal{U}(\frac{1}{2}\varepsilon_0,\supp(\mu))$, for any $x,y\in V$ with $d(x,y)<\rho_0$, for every $u\in C^{G_j}_\theta(y)$, $\|u\|=1$, $i\in\{F_j,G_j\}$, then
$$\|\pi_i(fy)(d_yf(u)) - \pi_i(fx)(d_xf(\widetilde{u}))\|\leq\frac{1}{10}\theta\varepsilon_0 a,$$
where $a=\displaystyle{\min_{\substack{x\in\overline{V}\\ u\in C^{G_j}_{\theta}(x)\\ \|u\|=1}}}\|\pi_{G_j}(fx)\big(d_xf(u)\big)\|$.
\end{itemize}

By Katok's entropy formula, for each $\delta\in(0,1)$ we have
\begin{eqnarray*}
\begin{aligned}
h_\mu(f)&=\lim_{\widetilde{\rho}\to 0}\liminf_{n\to\infty}\frac{1}{n}\log N(\mu,n,\widetilde{\rho},\delta)\\
&=\lim_{\widetilde{\rho}\to 0}\limsup_{n\to\infty}\frac{1}{n}\log N(\mu,n,\widetilde{\rho},\delta),
\end{aligned}
\end{eqnarray*}
where $N(\mu,n,\widetilde{\rho},\delta)$ denotes the minimal number of $(n,\widetilde{\rho})$-Bowen balls that are needed to cover a set of measure $\mu$ at least $1-\delta$.
Denote $\vartheta=\min\{\lambda_2(\mu)-\lambda_1(\mu),\lambda_3(\mu)-\lambda_2(\mu),\cdots,\lambda_l(\mu)-\lambda_{l-1}(\mu)\}$.
Fixing $\delta\in(0,1)$, for any $\varepsilon\in(0,\varepsilon_0)$ with $\lambda_l(\mu)-3\varepsilon>0$ and $\vartheta-8\varepsilon>0$, there exists $0<\rho_1<\min\{\rho_0,\frac{\varepsilon}{2}\}$ and a positive integer $N_4$ such that any $\widetilde{\rho}\in(0,\rho_1)$ and $n\geq N_4$ we have
\begin{equation}\label{minimalnumber}
N(\mu,n,\widetilde{\rho},\delta)\geq e^{[h_\mu(f)-\frac12\varepsilon]n}.
\end{equation}
By Lemma \ref{malyexoffN} and Lemma \ref{milyexoffN}, there exists a subset $\Omega\subseteq M$ with $\mu(\Omega)=1$, and for the previous $\varepsilon>0$, $\exists N_1(\varepsilon)>0$ and $N_2(\varepsilon)>0$ such that any $N\geq\max\{N_1(\varepsilon), N_2(\varepsilon),N_4\}$, any $x\in\Omega$, $j=1,2,\cdots,l-1$, we have
\begin{equation*}
\begin{aligned}
\lim_{k\to+\infty}\frac {1}{kN}\sum_{\iota=0}^{k-1}\log\|d_{f^{\iota N}x}f^N|_{F_j(f^{\iota N}x)}\|&<\lambda_j(\mu)+\varepsilon, \\
\lim_{k\to+\infty}\frac {1}{kN}\sum_{\iota=0}^{k-1}\log m(d_{f^{\iota N}x}f^N|_{G_j(f^{\iota N}x)})&>\lambda_{j+1}(\mu)-\varepsilon,\\
\lim_{k\to+\infty}\frac {1}{kN}\sum_{\iota=0}^{k-1}\log m(d_{f^{\iota N}x}f^{N}|_{F_j(f^{\iota N}x)})&>\lambda_1(\mu)-\varepsilon, \\
\lim_{k\to+\infty}\frac {1}{kN}\sum_{\iota=0}^{k-1}\log \|d_{f^{\iota N}x}f^{N}|_{G_j(f^{\iota N}x)}\|&<\lambda_l(\mu)+\varepsilon.\\
\end{aligned}
\end{equation*}
 Fixing any $L_0>\max\{N_1({\varepsilon}),N_2(\varepsilon),N_4\}$, by the Egornov theorem, for $\delta$ as above, there exists a compact subset $\Omega_\delta\subset\Omega$ with $\mu(\Omega_\delta)>1-\frac{\delta}{3}$ and  a positive integer $K_\varepsilon>1$ such that for any $x\in \Omega_\delta$ and $k\geq K_\varepsilon$,  $j=1,2,\cdots,l-1$, we have
\begin{equation}\label{lyexoff}
\begin{aligned}
\prod_{\iota =0}^{k-1}\|d_{f^{\iota L_0}x}f^{L_0}|_{F_j(f^{\iota L_0}x)}\|&\leq e^{kL_0[\lambda_j(\mu)+2\varepsilon]},\\
 \prod_{\iota =0}^{k-1} m(d_{f^{\iota L_0}x}f^{L_0}|_{G_j(f^{\iota L_0}x)})&\geq e^{kL_0[\lambda_{j+1}(\mu)-2\varepsilon]},
\end{aligned}
\end{equation}
\begin{equation}\label{lyexoff-1}
\begin{aligned}
\prod_{\iota=0}^{k-1} m(d_{f^{\iota L_0}x}f^{L_0}|_{F_j(f^{\iota L_0}x)})&\geq e^{kL_0[\lambda_1(\mu)-2\varepsilon]},\\
 \prod_{\iota=0}^{k-1} \|d_{f^{\iota L_0}x}f^{L_0}|_{G_j(f^{\iota L_0}x)}\|&\leq e^{kL_0[\lambda_l(\mu)+2\varepsilon]}.
\end{aligned}
\end{equation}
For the above $\varepsilon>0$, there exist small $0<\rho_2<\rho_1$ and $\zeta_0>0$ so that for every $\iota=1,2,\cdots,L_0$,
\begin{equation}\label{contoff}
\begin{aligned}
e^{-\varepsilon}\leq\frac{\|d_xf^\iota(u)\|}{\|d_yf^\iota(v)\|}\leq e^\varepsilon, e^{-\varepsilon}\leq\frac{m (d_xf^\iota(u) )}{m(d_yf^\iota(v))}\leq e^\varepsilon,
\end{aligned}
\end{equation}
\begin{equation}\label{contoff-1}
\begin{aligned}
e^{-\varepsilon}\leq\frac{\|d_xf^{-\iota}(u)\|}{\|d_yf^{-\iota}(v)\|}\leq e^\varepsilon, e^{-\varepsilon}\leq\frac{m (d_xf^{-\iota}(u) )}{m(d_yf^{-\iota}(v))}\leq e^\varepsilon,
\end{aligned}
\end{equation}
whenever $d(x,y)<\rho_2$ and $\angle(u,v)<\zeta_0$. For any $0<\zeta<\zeta_0$, there exist small $\rho_3\in(0,\rho_2)$ and small $\theta_0\in(0,1)$ such that if $x,y\in\overline{V}$ with $d(x,y)<\rho_3$, then $\angle(w_1,w_2)<\zeta$ for any $0\neq w_1\in C_{\theta_0}^i(x), 0\neq w_2\in C_{\theta_0}^i(y)$ and $i\in\{F_j, G_j\}$.

Pick a countable basis $\{\varphi_i\}_{i\geq 1}$ (nonzero) of  the space $C^0(M)$ of all continuous functions on $M$. Recall that the space of $f$-invariant probabilities $\mathcal{M}_f(M)$ can be endowed with the metric $d: \mathcal{M}_f(M)\times\mathcal{M}_f(M)\to [0,1]$,
$$d(\mu,\nu):=\sum_{j=1}^\infty 2^{-j}\frac{1}{2\|\varphi_j\|_\infty}\Big|\int\varphi_jd\mu-\int\varphi_jd\nu\Big|$$
where $\|\varphi\|_\infty:=\sup_{x\in M}|\varphi(x)|$. Let $J$ be a positive integer satisfying $\frac{1}{2^J}<\frac{\varepsilon}{8}$ and $\rho\in(0,\frac{\rho_3}{2})$ such that
\begin{equation}\label{continuous}
\begin{aligned}
|\varphi_j(x)-\varphi_j(y)|\leq\frac{\varepsilon}{4}\|\varphi_j\|_\infty
\end{aligned}
\end{equation}
for any $x,y$ with $d(x,y)\leq\rho$, $j=1,2,\cdots,J$.

Let $\chi:=\lambda(\mu)=\min_{1\leq i\leq l}|\lambda_i(\mu)|$. For $\rho>0$ and $\delta>0$ as above, by Lemma \ref{existofrectangle}, there exists a compact set $\Lambda_H=\Lambda_H(\rho,\delta,\frac{\chi}{2})$ with $\mu(\Lambda_H)>1-\frac{\delta}{3}$, a constant $\beta=\beta(\rho,\delta)>0$ and a finite collection of $(\rho,\beta,\frac{\chi}{2})$-rectangles $R(q_1),R(q_2),\cdots,R(q_t)$ with $q_j\in\Lambda_H$ so that $\Lambda_H\subseteq\cup_{j=1}^tB(q_j,\beta)$.

Let $\mathcal{P}$ be a finite measurable partition of $\Lambda_H\cap\Omega_\delta\cap\supp(\mu)$ so that $P(q_j)\subseteq B(q_j,\beta)$ for $j=1,2,\cdots,t$. We can take $\mathcal{P}$ as follows:
\begin{eqnarray*}
\begin{aligned}
P_1&=B(q_1,\beta)\cap\Lambda_H\cap\Omega_\delta\cap\supp(\mu),\\
P_k&=\big{(}B(q_k,\beta)\cap\Lambda_H\cap\Omega_\delta\cap\supp(\mu)\big{)}\setminus\big{(}\cup_{j=1}^{k-1}P_j\big{)}\ \mbox{for}\ k=2,3,\cdots,t.
\end{aligned}
\end{eqnarray*}
Denote
\begin{eqnarray*}
\begin{aligned}
\Lambda_{H, \delta,n}=&\Big\{x\in\Lambda_H\cap\Omega_\delta\cap\supp(\mu): f^kx\in\mathcal{P}(x)\ \mbox{for some integer}\ k\in[n,(1+\varepsilon)n)\ \mbox{and} \\ &\Big|\frac{1}{m}\sum_{i=0}^{m-1}\varphi_j(f^ix)-\int\varphi_jd\mu\Big|\leq\frac{\varepsilon}{4}\|\varphi_j\|_\infty\ \mbox{for any}\ m\geq n\ \mbox{and}\ j=1,2,\cdots,J\Big\}.
\end{aligned}
\end{eqnarray*}

\begin{lemma}
$\lim_{n\to\infty}\mu(\Lambda_{H,\delta,n})=\mu(\Lambda_H\cap\Omega_\delta\cap\supp(\mu)).$
\end{lemma}
\begin{proof} Let
\begin{eqnarray*}
\begin{aligned}
A_n&=\Big\{x\in\Lambda_H\cap\Omega_\delta\cap\supp(\mu): f^kx\in P(x)\ \mbox{for some integer}\ k\in[n,(1+\varepsilon)n)\Big\},\\
A_{n,j}&=\Big\{x\in P_j: f^kx\in P_j\ \mbox{for some integer}\ k\in[n,(1+\varepsilon)n)\Big\}.
\end{aligned}
\end{eqnarray*}
It is easy to see $A_n=\cup_{j=1}^tA_{n,j}$. For any $P_j$ with $\mu(P_j)>0$ and small $\tau>0$, denote
$$A_{n,j}^\tau=\Big\{x\in P_j: \mu(P_j)-\tau\leq\frac{1}{m}\sum_{j=0}^{m-1}\chi_{P_j}(f^jx)\leq\mu(P_j)+\tau\ \mbox{for any}\ m\geq n\Big\}.$$
By Birkhoff Ergodic Theorem, $\mu(\cup_{n\geq 1}A_{n,j}^\tau)=\mu(P_j)$. Since $A_{1,j}^\tau\subseteq A_{2,j}^\tau\subseteq\cdots\subseteq A_{n,j}^\tau\subseteq \cdots$,
we conclude $\lim_{n\to\infty}\mu(A_{n,j}^\tau)=\mu(P_j)$. For any $x\in A_{n,j}^\tau$,
\begin{eqnarray*}
\begin{aligned}
&\ \ \ \ Card\Big\{k\in[n,(1+\varepsilon)n): f^kx\in P_j\Big\}\\
&\geq Card\Big\{k\in[0,(1+\varepsilon)n): f^kx\in P_j\Big\} - Card\Big\{k\in[0,n): f^kx\in P_j\Big\}\\
&\geq \big[\mu(P_j)-\tau\big]\cdot(1+\varepsilon)n - \big[\mu(P_j)+\tau\big]n\\
&= n\cdot\big[\mu(P_j)\varepsilon-2\tau-\tau\varepsilon\big].
\end{aligned}
\end{eqnarray*}
Taking $0<\tau<\frac{\varepsilon}{2+\varepsilon}\min\big\{\mu(P_j): \mu(P_j)>0, j=1,2,\cdots,t\big\}$, therefore $$Card\big\{k\in[n,(1+\varepsilon)n): f^kx\in P_j\big\}>1$$ for $n$ large enough.
This yields that $x\in A_{n,j}$. Thus $A_{n,j}^\tau\subseteq A_{n,j}$. Then we have
\begin{equation}\label{limset}
\begin{aligned}
\lim_{n\to\infty}\mu(A_n)=\mu(\Lambda_H\cap\Omega_\delta\cap\supp(\mu)).
\end{aligned}
\end{equation}

Let
\begin{equation*}
\begin{aligned}
B_n=&\Big\{x\in\Lambda_H\cap\Omega_\delta\cap\supp(\mu): \Big|\frac{1}{m}\sum_{k=0}^{m-1}\varphi_j(f^kx)-\int\varphi_jd\mu\Big|\leq\frac{\varepsilon}{4}\|\varphi_j\|_\infty\\
&\mbox{for any}\ m\geq n\ \mbox{and}\ 1\leq j\leq J\Big\}.
\end{aligned}
\end{equation*}
Birkhoff Ergodic Theorem tells us that $\mu(\cup_{n\geq 1}B_n)=\mu(\Lambda_H\cap\Omega_\delta\cap\supp(\mu))$. Since $B_1\subseteq B_2\subseteq\cdots\subseteq B_n\subseteq\cdots$, we have $\lim_{n\to\infty}\mu(B_n)=\mu(\Lambda_H\cap\Omega_\delta\cap\supp(\mu))$. Combining with (\ref{limset}),
$$\lim_{n\to\infty}\mu(\Lambda_{H,\delta,n})=\lim_{n\to\infty}\mu(A_n\cap B_n)=\mu(\Lambda_H\cap\Omega_\delta\cap\supp(\mu)).$$
This shows the lemma.
\end{proof}

We proceed to prove the main theorem. Taking
\begin{equation}\label{star}
\begin{aligned}
N_5>\max\Big{\{}&\frac{1}{\varepsilon}, 2(K_\varepsilon+1) L_0, \frac{4\log Q}{\varepsilon}, \frac{4\log Q_1}{\varepsilon}, \frac{4\log Q_2}{\varepsilon}, \frac{4}{\varepsilon}\log t,\\
 &\frac{4(K_\varepsilon+1)L_0[\lambda_l(\mu)-3\varepsilon]}{\varepsilon}, \frac{2(K_\varepsilon+1)L_0 (\overline{\vartheta}-6\varepsilon)}{\varepsilon}, N_4\Big{\}}
\end{aligned}
\end{equation}
large enough with $\mu(\Lambda_{H,\delta,N_5})>\mu(\Lambda_H\cap\Omega_\delta\cap\supp(\mu))-\frac{\delta}{3}>1-\delta$ and $N_5\varepsilon<e^{\frac12N_5\varepsilon}$,
where
\begin{equation*}
\begin{aligned}
Q=\sup\Big\{&\prod_{\iota=0}^{k-1}\|d_{f^{\iota L_0}x}f^{L_0}|_{F_j(f^{\iota L_0}x)}\|, \prod_{\iota=0}^{k-1}m(d_{f^{\iota L_0}x}f^{L_0}|_{G_j(f^{\iota L_0}x)})^{-1},\\
&\prod_{\iota=0}^{k-1} m(d_{f^{\iota L_0}x}f^{L_0}|_{F_1(f^{\iota L_0}x)})^{-1}, \prod_{\iota=0}^{k-1}\|d_{f^{\iota L_0}x}f^{L_0}|_{G_{l-1}(f^{\iota L_0}x)}\|:\\
& k=1,2,\cdots, K_\varepsilon-1, \text{ and }j=1,2,\cdots,l-1, x\in\Omega_\delta\Big\},\\
Q_1=\sup\Big\{&1,m(d_xf|_{G_j(x)})^{-1},m(d_xf^2|_{G_j(x)})^{-1},\cdots,m(d_xf^{L_0-1}|_{G_j(x)})^{-1}, \\
&m(d_xf^{-1}|_{G_j(x)})^{-1},m(d_xf^{-2}|_{G_j(x)})^{-1},\cdots,m(d_xf^{-(L_0-1)}|_{G_j(x)})^{-1}:\\
&j=1,2, \cdots,l-1, \text{ and } x \text{ belongs to the }\varepsilon\text{-neighborhood of } \supp(\mu)\Big\},\\
Q_2=\sup\Big\{&1,\|d_xf|_{F_j(x)}\|,\cdots,\|d_xf^{L_0-1}|_{F_j(x)}\|,\|d_xf^{-1}|_{F_j(x)}\|,\cdots,\|d_xf^{-L_0+1}|_{F_j(x)}\|:\\
&j=1,2,\cdots,l-1,  \text{ and } x \text{ belongs to the }\varepsilon\text{-neighborhood of } \supp(\mu)\Big\},\\
\overline{\vartheta}=\max\Big\{&\lambda_2(\mu)-\lambda_1(\mu), \lambda_3(\mu)-\lambda_2(\mu), \cdots, \lambda_l(\mu)-\lambda_{l-1}(\mu)\Big\}.
\end{aligned}
\end{equation*}
Choose a maximal $(N_5,2\rho)$-separated subset $E\subseteq\Lambda_{H,\delta,N_5}$. Hence $\cup_{x\in E}B_{N_5}(x,2\rho)\supseteq\Lambda_{H,\delta,N_5}$. By (\ref{minimalnumber}),
\begin{eqnarray*}
\begin{aligned}
Card(E)&\geq N\Big(\mu,N_5,2\rho,1-\mu(\Lambda_{H,\delta,N_5})\Big)\\
&\geq N(\mu,N_5,2\rho,\delta)\\
&\geq e^{[h_\mu(f)-\frac12\varepsilon]N_5}.
\end{aligned}
\end{eqnarray*}
For $N_5\leq k< (1+\varepsilon)N_5$, let $\Delta_k=\big\{x\in E: f^kx\in P(x)\big\}$. Take $m\in[N_5,(1+\varepsilon)N_5)$ with $Card( \Delta_m)=\max\big\{Card (\Delta_k): N_5\leq k<(1+\varepsilon)N_5\big\}$. Then
\begin{eqnarray*}
\begin{aligned}
Card (\Delta_m)&\geq \frac{1}{N_5\varepsilon}Card (E)\\
&\geq \frac{1}{N_5\varepsilon}e^{[h_\mu(f)-\frac12\varepsilon]\cdot N_5}\\
&\geq e^{[h_\mu(f)-\varepsilon]\cdot N_5}.
\end{aligned}
\end{eqnarray*}
Choose $P\in\mathcal{P}$ with
$$Card(\Delta_m\cap P)=\max\Big\{Card(\Delta_m\cap P_k): 1\leq k\leq t\Big\}.$$
Thus $Card(\Delta_m\cap P)\geq\frac{1}{t}Card(\Delta_m)\geq\frac{1}{t}e^{[h_\mu(f)-\varepsilon]\cdot N_5}$.
Possibly neglecting some points in $E$ then we can guarantee that
\begin{equation}\label{uplowbounded}
\frac{1}{t}e^{[h_\mu(f)-\varepsilon]N_5}\leq Card(\Delta_m\cap P)\leq Card(\Delta_m)\leq Card(E) \leq e^{[h_\mu(f)+\varepsilon]N_5}.
\end{equation}
By the definition of the partition $\mathcal{P}$, there exists $q\in\{q_1,q_2,\cdots,q_t\}$ such that $P\subseteq B(q,\beta)\cap\Lambda_H\cap\Omega_\delta\cap\supp(\mu)$. For any $x\in \Delta_m\cap P$,
since $x,f^mx\in B(q,\beta)\cap\Lambda_H$, by Definition \ref{smallrectangle},
$$C\Big(x,R(q)\cap f^{-m}R(q)\Big)\ \mbox{and}\ f^m\Big(C\big(x,R(q)\cap f^{-m}R(q)\big)\Big)$$
are admissible $s$-rectangle and $u$-rectangle in $R(q)$ respectively for some number $\widetilde{\lambda}>0$.

Notice that for any $x_1,x_2\in \Delta_m\cap P$ with $x_1\neq x_2$, $$C\Big(x_1,R(q)\cap f^{-m}R(q)\Big)\cap C\Big(x_2,R(q)\cap f^{-m}R(q)\Big)=\emptyset.$$
In fact, suppose $y\in C\Big(x_1,R(q)\cap f^{-m}R(q)\Big)\cap C\Big(x_2,R(q)\cap f^{-m}R(q)\Big)$, by Definition \ref{smallrectangle}, we obtain
$d(f^kx_j,f^ky)\leq\rho\ \mbox{for}\ k=0,1,2,\cdots,m, \text{ and } j=1,2.$
Then we have
$$d_m(x_1,x_2)\leq d_m(x_1,y)+d_m(y,x_2)\leq 2\rho.$$
Since $\Delta_m\cap P$ is a $(N_5,2\rho)$-separate set, $$d_m(x_1,x_2)\geq d_n(x_1,x_2)>2\rho$$ which contradicts $d_m(x_1,x_2)\leq 2\rho$. This implies that there are $Card(\Delta_m\cap P)$ disjoint
admissible $s$-rectangles which mapped under $f^m$ onto $Card(\Delta_m\cap P)$ disjoint admissible $u$-rectangles.

Let $$\Lambda^*=\cap_{n\in\mathbb{Z}}f^{-nm}\Big(\cup_{x\in \Delta_m\cap P}C\big(x,R(q)\cap f^{-m}R(q)\big)\Big).$$ By the construction, $\Lambda^*$ is locally maximal with respect to $f^m$ and to the closed neighborhood $\cup_{i=1}^tR(q_i)$. It also implies that $f^m|_{\Lambda^*}$ is topologically conjugate to a full two-side shift in the symbolic space with $Card(\Delta_m\cap P)$ symbols. Let $\Lambda=\Lambda^*\cup f(\Lambda^*)\cup\cdots\cup f^{m-1}(\Lambda^*)$.
We claim that $\Lambda$ is in a neighborhood of $\supp(\mu)$.

\begin{claim}\label{hyperinne}
$\Lambda \subseteq \mathcal{U}(\rho,\supp(\mu))$.
\end{claim}
\begin{proof}By the construction of $\Lambda$ and the definition of $R(q)$, for any $y\in\Lambda$, $\exists x\in \Delta_m\cap P$,
$$d(y,f^kx)\leq\rho \text{ for some integer } k\in[0,m-1].$$
Since $x\in \Delta_m\cap P\subseteq\Lambda_H\cap\Omega_\delta\cap\supp(\mu)\subseteq\supp(\mu)$ and $\supp(\mu)$ is $f$-invariant,
$\Lambda$ is contained in the $\rho$-neighborhood of $\supp(\mu)$.
\end{proof}


It follows from Lemma \ref{invariantofcone} and Corollary \ref{corofle} below that there is a dominated splitting, corresponding to Oseledec subspace on horseshoes, and the approximation of Lyapunov exponents by horseshoes.

\begin{lemma}\label{invariantofcone}
 Let $\varepsilon$, $\theta_0$, $\rho$, $m$ and $\Lambda^*$ be as above. For any small $\theta\in(0,\theta_0)$, there exists $0<\eta<1$ such that for every $y\in\Lambda^*$, $j\in\{1,2,\cdots,l-1\}$ and $n\in\mathbb{Z}$,
\begin{equation*}
\begin{aligned}
d_{f^{nm}y}f^m C^{G_j}_{\theta}(f^{nm}y)\subseteq C^{G_j}_{\eta^m\theta}(f^{(n+1)m}y),\\
d_{f^{nm}y}f^{-m} C^{F_j}_{\theta}(f^{nm}y)\subseteq C^{F_j}_{\eta^m\theta}(f^{(n-1)m}y).
\end{aligned}
\end{equation*}
Moreover for any nonzero vectors $v\in C^{G_j}_{\theta}(f^{nm}y)$, $w\in C^{F_j}_{\theta}(f^{nm}y)$,
\begin{equation*}
\begin{aligned}
\|d_{f^{nm}y}f^m(v)\|&\geq e^{[\lambda_{j+1}(\mu)-6\varepsilon]m}\|v\|, \\
\|d_{f^{nm}y}f^{-m}(w)\|&\geq e^{[-\lambda_{j}(\mu)-6\varepsilon]m}\|w\|.
\end{aligned}
\end{equation*}
Furthermore,  for any nonzero vectors $v\in C^{G_{l-1}}_{\theta}(f^{nm}y)$, $w\in C^{F_1}_{\theta}(f^{nm}y)$,
\begin{equation*}
\begin{aligned}
\|d_{f^{nm}y}f^m(v)\|&\leq e^{[\lambda_l(\mu)+6\varepsilon]m}\|v\|, \\
\|d_{f^{nm}y}f^{-m}(w)\|&\leq e^{[-\lambda_{1}(\mu)+6\varepsilon]m}\|w\|.
\end{aligned}
\end{equation*}
\end{lemma}

\begin{proof}
We only prove the statements for $G_j$ with respective to $f^m$, since the other statements for $F_j$ with respective to $f^{-m}$ can be proven in a similar fashion.

First of all, we prove the $df^m$-invariance of the cones of $G_j$.
For any $y\in\Lambda$, by Claim \ref{hyperinne}, there is $x\in\supp(\mu)$ such that $d(x,y)<\rho$. Since $x\in\supp(\mu)$, $T_xM=F_j(x)\oplus_<G_j(x)$ is dominated. By Remark \ref{conds}, choosing a approximate norm with $N=1$, for each pair of unit vectors $u_F\in F_j(x)$ and $v_G\in G_j(x)$
\begin{equation}\label{1}
\|d_xf(u_F)\|\leq\frac12\|d_xf(v_G)\|.
\end{equation}
As $d(x,y)<\rho<\rho_0$, applying $(3)$ and $(4)$ we have for every unit vector  $v\in C^{G_j}_{\theta}(y)$,
\[ \widetilde{v}\in C_{\frac32\theta}^{G_j}(x) \text{ and } \|\pi_i(fy)\big(d_yf(v)\big)- \pi_i(fx)\big(d_xf(\widetilde{v})\big)\|\leq\frac1{10}\theta\varepsilon_0 a \]
where $i\in\{F_j,G_j\}$. Therefore
\begin{equation*}
\begin{aligned}
&\|\pi_{F_j}(fy)\big(  d_yf(v)  \big)\|\\
\leq &\| \pi_{F_j}(fy)\big(  d_yf(v)  \big) - \pi_{F_j}(fx)\big(  d_xf(\widetilde{v})  \big)\| + \|\pi_{F_j}(fx)\big(  d_xf(\widetilde{v})  \big)\|\\
\leq & \frac1{10}\theta\varepsilon_0a + \|d_xf(\widetilde{v}_F)\|\\
\leq & \frac1{10}\theta\varepsilon_0a + \frac12 \|d_xf(\widetilde{v}_G)\| \cdot \frac{\|\widetilde{v}_F\|}{\|\widetilde{v}_G\|}\\
\leq & \frac1{10}\theta\varepsilon_0 a + \frac34\theta \|d_xf(\widetilde{v}_G)\|\\
= & \frac1{10}\theta\varepsilon_0 a + \frac34\theta \|\pi_{G_j}(fx)\big(d_xf(\widetilde{v})\big)\|\\
\leq & \frac1{10}\theta\varepsilon_0 a + \frac34\theta \Big(\|\pi_{G_j}(fx)\big(d_xf(\widetilde{v})\big) - \pi_{G_j}(fy)\big(d_yf(v)\big)\|  + \|\pi_{G_j}(fy)\big(d_yf(v)\big) \| \Big)\\
\leq &  \frac1{10}\theta\varepsilon_0a + \frac34\theta \cdot\frac1{10}\theta\varepsilon_0 a + \frac34\theta \|\pi_{G_j}(fy)\big(d_yf(v)\big) \| \\
\leq & \Big( \frac1{10}\varepsilon_0 +   \frac34 \cdot\frac1{10}\theta\varepsilon_0 +  \frac34  \Big)\theta \|\pi_{G_j}(fy)\big(d_yf(v)\big) \| \\
\leq & \eta\theta\cdot \|\pi_{G_j}(fy)\big(d_yf(v)\big) \|,
\end{aligned}
\end{equation*}
where $\widetilde{v}=\widetilde{v}_F+\widetilde{v}_G$, $\widetilde{v}_F\in F_j(x)$, $\widetilde{v}_G\in G_j(x)$ and $\eta=\frac{37}{40}\in(0,1)$. This yields that for any $y\in\Lambda$,
\begin{equation}\label{2}
d_yfC^{G_j}_{\theta}(y)\subseteq C^{G_j}_{\eta\theta}(fy).
\end{equation}
Similarly we can prove $d_yf^{-1}C^{F_j}_{\theta}(y)\subseteq C^{F_j}_{\eta\theta}(f^{-1}y)$ for any $y\in\Lambda$. Then it is obvious that
 \begin{equation*}
\begin{aligned}
d_{f^{nm}y}f^m C^{G_j}_{\theta}(f^{nm}y)\subseteq C^{G_j}_{\eta^m\theta}(f^{(n+1)m}y),\\
d_{f^{nm}y}f^{-m} C^{F_j}_{\theta}(f^{nm}y)\subseteq C^{F_j}_{\eta^m\theta}(f^{(n-1)m}y)
\end{aligned}
\end{equation*}
for any $y\in\Lambda$ and $n\in\mathbb{Z}$. Since $\Lambda^*\subseteq\Lambda$, we complete the proof of the first statement of this lemma.

To prove the second statement of the lemma, we first define the families of sets $\widetilde{F}_j$ and $\widetilde{G}_j$ which are in the cone of $F_j$ and $G_j$ respectively.
For any $y\in\Lambda^*$, let
 \[ \widetilde{F}_j(y) = \cap_{n=0}^{\infty} d_{f^{nm}y}f^{-nm} C_\theta^{F_{j}}(f^{nm}y) \text{ and } \widetilde{G}_j(y) = \cap_{n=0}^{\infty} d_{f^{-nm}y}f^{nm} C_\theta^{G_{j}}(f^{-nm}y)  \]
for $j=1,2,\cdots,l-1$. For any $z\in (f\Lambda^* \cup f^2\Lambda^*\cup\cdots\cup f^{m-1}\Lambda^*)\setminus\Lambda^*$, then there exists $k\in\{1,2,\cdots,m-1\}$ such that
$z\in f^k\Lambda^*$ but $z\notin \Lambda^*\cup\cdots\cup f^{k-1}\Lambda^*$. Define
\[  \widetilde{F}_j(z) = d_{f^{-k}z}f^k \widetilde{F}_j(f^{-k}z) \text{ and } \widetilde{G}_j(z) = d_{f^{-k}z}f^k \widetilde{G}_j(f^{-k}z). \]
Since $\Lambda=\Lambda^*\cup f\Lambda^*\cup\cdots\cup f^{m-1}\Lambda^*$ and $\Lambda^*$ is $f^m$-invariant, it yields that the families of sets $\widetilde{F}_j$ and $ \widetilde{G}_j$ are $df$-invariant. By Lemma \ref{ds} below, the splitting $T_\Lambda M=\widetilde{F}_j\oplus\widetilde{G}_j$ is dominated on $\Lambda$. Therefore the splitting $T_zM=\widetilde{F}_j(z) \oplus \widetilde{G}_j(z)$ is continuous for any $z\in\Lambda$.
From the construction of $\widetilde{F}_j$, $\widetilde{G}_j$ and (\ref{2}) we obtain
\begin{equation}\label{biaohao}
\widetilde{F}_j(z)\subseteq C_\theta^{F_j}(z) \text{ and } \widetilde{G}_j(z)\subseteq C_\theta^{G_j}(z) \text{ for any } z\in\Lambda.
\end{equation}
Since $\theta_0$ is small, $\theta\in(0,\theta_0)$ and the splitting  $T_zM=F_j(z) \oplus G_j(z)$ is continuous for any $z\in\Lambda$, then there exists a constant $\kappa>0$ (independent of $z\in\Lambda$) satisfying $\kappa\theta<\frac78$ such that
for any nonzero vector $v\in C^{G_j}_{\theta}(z)$, we have $\|v^s\|\leq \kappa\theta\|v^u\|$
where $v=v^s+v^u$, $v^s\in \widetilde{F}_j(z)$, $v^u\in \widetilde{G}_j(z)$.
For any $y\in\Lambda^*$, any $n\in\mathbb{Z}$, there exists $x_n\in \Delta_m\cap P$ such that $f^{nm}y\in C(x_n, R(q)\cap f^{-m}R(q))$. We may as well assume $n=0$ and denote $x_0=x$. The proof of $n\neq0$ is parallel to that of $n=0$. By definition \ref{smallrectangle}, $d(f^kx,f^ky)\leq\rho$ for $k=0,1,\cdots,m$.
For any $0\neq v\in C_\theta^{G_j}(y)$, then $v=v^s+v^u$, $v^s\in \widetilde{F}_j(y)$, $v^u\in \widetilde{G}_j(y)$ with $\|v^s\|\leq \kappa\theta\|v^u\|$. Therefore
\begin{equation*}
\begin{aligned}
\|d_yf^m(v^s)\|&\leq\|d_yf^m|_{\widetilde{F}_j(y)}\|\cdot\|v^s\|\\
&\leq\Big(\prod_{\iota=0}^{p-1}\|d_{f^{\iota L_0}y}f^{L_0}|_{\widetilde{F}_j(f^{\iota L_0}y)}\|\Big)\cdot Q_2 e^\varepsilon\cdot\|v^s\|
\end{aligned}
\end{equation*}
where $m=pL_0+q$, $p,q\in\mathbb{N}$ and $0\leq q<L_0$.
Lemma \ref{ds} and Remark \ref{conds} tell us that $\widetilde{F}_j(z)=F_j(z)$ and $\widetilde{G}_j(z)=G_j(z)$ if $z\in\Lambda\cap\supp(\mu)$.
For $\iota=0,1,2,\cdots,p-1$, we have
$ \widetilde{F}_j(f^{\iota L_0}y)\subseteq C_\theta^{F_j}(f^{\iota L_0}y)$.
Since $F_j(f^{\iota L_0}x)\subseteq C_\theta^{F_j}(f^{\iota L_0}x)$ and $m> K_\varepsilon L_0$, combining (\ref{contoff}) and (\ref{lyexoff}) we obtain
\begin{equation}\label{4}
\begin{aligned}
\|d_yf^m(v^s)\|&\leq\Big(\prod_{\iota=0}^{p-1}\|d_{f^{\iota L_0}x}f^{L_0}|_{F_j(f^{\iota L_0}x)}\|\Big)\cdot e^{p\varepsilon}\cdot Q_2 e^\varepsilon\cdot \|v^s\|\\
&\leq e^{pL_0[\lambda_j(\mu)+2\varepsilon]}\cdot e^{pL_0\varepsilon}\cdot Q_2\cdot\|v^s\|\\
&\leq \kappa\theta \cdot Q_2\cdot e^{pL_0[\lambda_j(\mu)+3\varepsilon]}\cdot\|v^u\|
\end{aligned}
\end{equation}
and
\begin{equation*}
\begin{aligned}
\|v^u\|&=\|d_{f^{m}y}f^{-m}\big(d_yf^m(v^u)\big)\|\\
&\leq\|d_{f^my}f^{-m}|_{\widetilde{G}_j(f^my)}\|\cdot\|d_yf^m(v^u)\|\\
&=m(d_yf^m|_{\widetilde{G}_j(y)})^{-1}\cdot\|d_yf^m(v^u)\|\\
&\leq \Big(\prod_{\iota=0}^{p-1}m (d_{f^{\iota L_0}y}f^{L_0}|_{\widetilde{G}_j(f^{\iota L_0}y)})^{-1}\Big)\cdot Q_1 e^\varepsilon\cdot\|d_yf^m(v^u)\|
\end{aligned}
\end{equation*}
\begin{equation*}
\begin{aligned}
&\leq \Big(\prod_{\iota=0}^{p-1}m (d_{f^{\iota L_0}x}f^{L_0}|_{G_j(f^{\iota L_0}x)})^{-1}\Big)\cdot e^{p\varepsilon}\cdot Q_1 e^\varepsilon\cdot\|d_yf^m(v^u)\|\\
&\leq e^{-pL_0[\lambda_{j+1}(\mu)-2\varepsilon]}\cdot e^{pL_0\varepsilon}\cdot Q_1\cdot \|d_yf^m(v^u)\|\\
&= e^{pL_0[-\lambda_{j+1}(\mu)+3\varepsilon]}\cdot Q_1\cdot \|d_yf^m(v^u)\|.
\end{aligned}
\end{equation*}
This implies
\begin{equation*}\label{unequalityofsu}
\begin{aligned}
\|d_yf^m(v^s)\|\leq \kappa\theta Q_1\cdot Q_2\cdot e^{pL_0[\lambda_j(\mu)-\lambda_{j+1}(\mu)+6\varepsilon]}\cdot\|d_yf^m(v^u)\|.
\end{aligned}
\end{equation*}
(\ref{star}) and $m\geq N_5$ tell us that
\[m > \max\Big\{\frac{4\log Q_1}{\varepsilon}, \frac{4\log Q_2}{\varepsilon}, \frac{L_0[\lambda_l(\mu)-3\varepsilon]}{\varepsilon}, \frac{L_0 (\overline{\vartheta}-6\varepsilon)}{\varepsilon}\Big\}.\] It follows that
\begin{equation*}
\begin{aligned}
\frac{\|d_yf^m(v)\|}{\|v\|}&=\frac{\|v^u\|}{\|v\|}\cdot\frac{\|d_yf^m(v)\|}{\|v^u\|}\\
&\geq\frac{\|v^u\|}{\|v\|}\cdot\frac{\|d_yf^m(v^u)\|-\|d_yf^m(v^s)\|}{\|v^u\|}\\
&\geq\frac{1}{1+(\kappa\theta)^2}\cdot\Big[ 1-\kappa\theta Q_1 Q_2\cdot e^{pL_0[\lambda_j(\mu) - \lambda_{j+1}(\mu)+6\varepsilon]}\Big]\cdot\frac{\|d_yf^m(v^u)\|}{\|v^u\|}\\
&\geq\frac{1}{2}\cdot\Big\{1- e^{m[\lambda_j(\mu)-\lambda_{j+1}(\mu)+8\varepsilon]}\Big\}\cdot e^{m[\lambda_{j+1}(\mu)-5\varepsilon]}.
\end{aligned}
\end{equation*}
Since $\lambda_j(\mu)<\lambda_{j+1}(\mu)$ and $\varepsilon>0$ small enough with $-\vartheta+8\varepsilon<0$,
then
\[ e^{m[\lambda_j(\mu)-\lambda_{j+1}(\mu)+8\varepsilon]}\] can be small enough for $m$ large enough.
Thus
\begin{equation*}
\begin{aligned}
\frac{\|d_yf^m(v)\|}{\|v\|}\geq e^{m[\lambda_{j+1}(\mu)-6\varepsilon]}.
\end{aligned}
\end{equation*}
We can prove for any nonzero vector $w\in C^{F_j}_{\theta}(y)$,
\begin{equation*}
\begin{aligned}
\|d_{y}f^{-m}(w)\|&\geq e^{m[-\lambda_{j}(\mu)-6\varepsilon]}\|w\|
\end{aligned}
\end{equation*}
by the same way. Therefore we complete the second statement of the lemma.

Last but not the least, we prove the third statement of the lemma. For every $y\in\Lambda^*$, $0\neq v\in C_\theta^{G_{l-1}}(y)$ with $v=v^s+v^u$, $v^s\in \widetilde{F}_{l-1}(y)$, $v^u\in \widetilde{G}_{l-1}(y)$ with $\|v^s\|\leq \kappa\theta\|v^u\|$. Since $m>\frac{4\log Q_2}{\varepsilon}$, by (\ref{4}), we conclude
\begin{equation*}
\begin{aligned}
\|d_yf^m(v^s)\|&\leq e^{pL_0[\lambda_{l-1}(\mu)+3\varepsilon]}\cdot Q_2 \cdot \|v^s\|\\
&\leq e^{m[\lambda_{l}(\mu)+4\varepsilon]} \|v^s\|.
\end{aligned}
\end{equation*}
By (\ref{contoff}), (\ref{lyexoff-1}) and $m>\frac{4\log Q_1}{\varepsilon}$, we have
\begin{equation*}
\begin{aligned}
\|d_yf^m(v^u)\|&\leq \|d_yf^m|_{\widetilde{G}_{l-1}(y)}\|\cdot\|v^u\|\\
&\leq\prod_{\iota=0}^{p-1} \|d_{f^{\iota L_0}y}f^m|_{\widetilde{G}_{l-1}(f^{\iota L_0}y)}\|\cdot Q_1e^\varepsilon\cdot\|v^u\|
\end{aligned}
\end{equation*}
\begin{equation*}
\begin{aligned}
&\leq\prod_{\iota=0}^{p-1} \|d_{f^{\iota L_0}x}f^m|_{G_{l-1}(f^{\iota L_0}x)}\|\cdot e^{p\varepsilon} Q_1e^\varepsilon\cdot\|v^u\|\\
&\leq e^{pL_0[\lambda_l(\mu)+3\varepsilon]}Q_1\cdot\|v^u\|\\
&\leq e^{m[\lambda_l(\mu)+4\varepsilon]}\cdot\|v^u\|.
\end{aligned}
\end{equation*}
It follows that
\begin{equation*}
\begin{aligned}
\frac{\|d_yf^m(v)\|}{\|v\|} &\leq \frac{\|d_yf^m(v^u)\|+\|d_yf^m(v^s)\|}{\|v\|}\\
&\leq\frac{\|v^u\|}{\|v\|}\cdot \frac{\|d_yf^m(v^u)\|+\|d_yf^m(v^s)\|}{\|v^u\|}\\
&\leq \frac{1}{1-\kappa\theta} \frac{\|d_yf^m(v^u)\|+\|d_yf^m(v^s)\|}{\|v^u\|}\\
&\leq \frac{1}{1-\kappa\theta} (1+\kappa\theta)e^{m[\lambda_l(\mu)+4\varepsilon]}\\
&\leq e^{m[\lambda_l(\mu)+6\varepsilon]},
\end{aligned}
\end{equation*}
the last inequality is because that $m$ is large enough.
Similarly we can also prove
$ \|d_{f^{nm}y}f^{-m}(w)\|\leq e^{[-\lambda_{1}(\mu)+6\varepsilon]m}\|w\|$
for any nonzero vector $w\in C^{F_1}_{\theta}(f^{nm}y)$.
Therefore the proof of Lemma \ref{invariantofcone} is completed.
\end{proof}

\begin{lemma}\label{ds}
The splitting $T_\Lambda M=\widetilde{F}_j\oplus\widetilde{G}_j$ is dominated on $\Lambda$ for $j=1,2,\cdots,l-1$.
\end{lemma}
\begin{proof}
For any $z\in\Lambda$, there is $x\in\supp(\mu)$ such that $d(x,z)<\rho$. Since $x\in\supp(\mu)$, $T_xM=F_j(x)\oplus_<G_j(x)$ is dominated. By (\ref{1})
for small $\theta\in(0,\theta_0)$, each pair of unit vectors $\overline{u}\in C_{\frac32\theta}^{F_j}(x)$ and $\overline{v}\in C_{\frac32\theta}^{G_j}(x)$,
\begin{equation*}
 \|d_x f(\overline{u})\|\leq\frac58\|d_x f(\overline{v})\|.
\end{equation*}
For any unit vectors $u\in\widetilde{F}_j(z)$ and $v\in\widetilde{G}_j(z)$, from (\ref{biaohao}) we obtain
\[ u\in\widetilde{F}_j(z)\subseteq C_\theta^{F_j}(z) \text{ and } v\in\widetilde{G}_j(z)\subseteq C_\theta^{G_j}(z). \]
Since $d(z,x)<\rho$, by $(3)$, we have $\widetilde{u}\in C_{\frac32\theta}^{F_j}(x)$ and $\widetilde{v}\in C_{\frac32\theta}^{G_j}(x)$.
Combining $(4)$, we conclude
\[ \frac{\|d_zf(u)\|}{\|d_zf(v)\|}\leq \frac{\|d_xf(\widetilde{u})\|+\frac15\theta\varepsilon_0a}{\|d_xf(\widetilde{v})\|-\frac15\theta\varepsilon_0a}. \]
Since $\varepsilon_0$ is small, it yields that
\begin{equation}\label{star2}
\frac{\|d_zf(u)\|}{\|d_zf(v)\|}\leq\frac34.
\end{equation}

By construction of $\widetilde{F}_j(z), \widetilde{G}_j(z)$ and the continuity of the decomposition $T_zM=F_j(z)\oplus G_j(z)$, it follows that $\widetilde{F}_j(z)$ and $\widetilde{G}_j(z)$
contain two subspaces $\overline{F}_j(z)$ and $\overline{G}_j(z)$, respectively of the same dimension as that of $F_j(z)$ and $G_j(z)$.
Since $\widetilde{F}_j(z)\cap\widetilde{G}_j(z)=\{0\}$, $\overline{F}_j(z)\cap\overline{G}_j(z)=\{0\}$. Therefore $T_zM=\overline{F}_j(z)\oplus\overline{G}_j(z)$. We claim that $\overline{F}_j(z)=\widetilde{F}_j(z)$. In fact,
let $w\in\widetilde{F}_j(z)$ with $w=w_{\overline{F}_j}+w_{\overline{G}_j}$, $w_{\overline{F}_j}\in\overline{F}_j(z)$, $w_{\overline{G}_j}\in\overline{G}_j(z)$ and $\|w\|=1$. If $w_{\overline{G}_j}\neq0$, by (\ref{star2}), we have for any $n\geq 1$
\[ \frac{\|d_zf^{n}(w_{\overline{F}_j})\|}{\|w_{\overline{F}_j}\|}\leq(\frac34)^{n}\frac{\|d_zf^{n}(w_{\overline{G}_j})\|}{\|w_{\overline{G}_j}\|} \text{ and }
\|d_zf^{n}(w)\|\leq(\frac34)^{n}\frac{\|d_zf^{n}(w_{\overline{G}_j})\|}{\|w_{\overline{G}_j}\|}, \]
because $w_{\overline{F}_j}\in\overline{F}_j(z)\subseteq\widetilde{F}_j(z)$, $w\in\widetilde{F}_j(z)$ and $w_{\overline{G}_j}\in\overline{G}_j(z)\subseteq\widetilde{G}_j(z)$. Therefore
\begin{eqnarray*}
\begin{aligned}
\|w_{\overline{G}_j}\|&\leq(\frac34)^n\frac{\|d_zf^n(w_{\overline{G}_j})\|}{\|d_zf^n(w)\|}\\
&\leq(\frac34)^n\frac{\|d_zf^n(w_{\overline{G}_j})\|}{\|d_zf^n(w_{\overline{G}_j})\|-\|d_zf^n(w_{\overline{F}_j})\|}\\
&\leq(\frac34)^n\frac{1}{1 - \frac{\|d_zf^n(w_{\overline{F}_j})\|}{\|d_zf^n(w_{\overline{G}_j})\|}}\\
&\leq(\frac34)^n\frac{1}{1 - (\frac{3}{4})^n\frac{\|w_{\overline{F}_j}\|}{\|w_{\overline{G}_j}\|}}.
\end{aligned}
\end{eqnarray*}
Let $n\to\infty$, it follows that $w_{\overline{G}_j}=0$. Thus $\overline{F}_j(z)=\widetilde{F}_j(z)$.
Similarly we have $\overline{G}_j(z)=\widetilde{G}_j(z)$.
Therefore the splitting $T_{\Lambda}M=\widetilde{F}_j\oplus\widetilde{G}_j$ is dominated.
\end{proof}

The following result is extended Lemma \ref{invariantofcone} to $\Lambda$.
\begin{corollary}\label{corofle}
 Let $\varepsilon$, $\theta_0$, $\rho$, $m$ and $\Lambda$ be as above. For any small $\theta\in(0,\theta_0)$, there exists $0<\eta<1$ such that for every $z\in\Lambda$, $j\in\{1,2,\cdots,l-1\}$ and $n\in\mathbb{Z}$,
\begin{equation*}
\begin{aligned}
d_{f^{nm}z}f^m C^{G_j}_{\theta}(f^{nm}z)\subseteq C^{G_j}_{\eta^m\theta}(f^{(n+1)m}z),\\
d_{f^{nm}z}f^{-m} C^{F_j}_{\theta}(f^{nm}z)\subseteq C^{F_j}_{\eta^m\theta}(f^{(n-1)m}z).
\end{aligned}
\end{equation*}
Moreover for any nonzero vectors $v\in C^{G_j}_{\theta}(f^{nm}z)$, $w\in C^{F_j}_{\theta}(f^{nm}z)$,
\begin{equation*}
\begin{aligned}
\|d_{f^{nm}z}f^m(v)\|&\geq e^{[\lambda_{j+1}(\mu)-6\varepsilon]m}\|v\|, \\
\|d_{f^{nm}z}f^{-m}(w)\|&\geq e^{[-\lambda_{j}(\mu)-6\varepsilon]m}\|w\|.
\end{aligned}
\end{equation*}
Furthermore, for any nonzero vectors $v\in C^{G_{l-1}}_{\theta}(f^{nm}z)$, $w\in C^{F_1}_{\theta}(f^{nm}z)$,
\begin{equation*}
\begin{aligned}
\|d_{f^{nm}z}f^m(v)\|&\leq e^{[\lambda_l(\mu)+6\varepsilon]m}\|v\|, \\
\|d_{f^{nm}z}f^{-m}(w)\|&\leq e^{[-\lambda_{1}(\mu)+6\varepsilon]m}\|w\|.
\end{aligned}
\end{equation*}
\end{corollary}
\begin{proof}
We only prove the statements for $G_j$ with respective to $f^m$, since the other statements for $F_j$ with respective to $f^{-m}$ can be proven in a similar fashion.
The proof of the first statement is identical to that of Lemma \ref{invariantofcone}. It suffices to prove for any
$0\neq v\in C^{G_j}_{\theta}(f^{nm}z)$, $j=1,2,\cdots,l-1$,
\[ \|d_{f^{nm}z}f^m(v)\|\geq e^{[\lambda_{j+1}(\mu)-6\varepsilon]m}\|v\|\]
and for any $0\neq v\in C^{G_{l-1}}_{\theta}(f^{nm}z)$,
\[ \|d_{f^{nm}z}f^m(v)\|\leq e^{[\lambda_l(\mu)+6\varepsilon]m}\|v\|. \]

For every $z\in\Lambda$, there exist $y\in\Lambda^*$ and $k\in\{0,1,2,\cdots,m-1\}$ such that $z=f^ky$. As $y\in\Lambda^*$, for any $n\geq 0$ there
is $x_n\in\Delta_m\cap P$ such that
\[d(f^Kx_n,f^K(f^{nm}y))\leq\rho, \text{ for }K=0,1,2,\cdots,m.\]
We may as well assume $n=0$. The proof of $n\neq0$ is parallel to that of $n=0$.
For any $0\neq v\in C_\theta^{G_j}(z)$, then $v=v^s+v^u$ with $v^s\in\widetilde{F}_j(z)$, $v^u\in\widetilde{G}_j(z)$ and $\|v^s\|\leq\kappa\theta\|v^u\|$.
Thus
\begin{equation}\label{q}
\begin{aligned}
\|d_zf^m(v^s)\|&=\|d_{f^{m-k}z}f^k\Big( d_zf^{m-k}(v^s)\Big)\|\\
&\leq \|d_{f^{m-k}z}f^k|_{\widetilde{F}_j(f^{m-k}z)}\|\cdot \|d_zf^{m-k}|_{\widetilde{F}_j(z)}\| \cdot \|v^s\|
\end{aligned}
\end{equation}
and
\begin{equation}\label{r}
\begin{aligned}
\|v^u\|&\leq m(d_zf^m|_{\widetilde{G}_j(z)})^{-1}\cdot \|d_zf^m(v^u)\|\\
&\leq m(d_zf^{m-k}|_{\widetilde{G}_j(z)})^{-1}\cdot m(d_{f^{m-k}z}f^k|_{\widetilde{G}_j(f^{m-k}z)})^{-1}\cdot  \|d_zf^m(v^u)\|.
\end{aligned}
\end{equation}
It follows from (\ref{star}) and $m\geq N_5$ that $\max\{m-k,k\}\geq K_\varepsilon L_0$.
Let $m-k=p_1L_0+q_1$ and $k=p_2L_0+q_2$ where $p_i, q_i\in\mathbb{N}$, $0\leq q_i<L_0$ and $i=1,2$.
If $m-k\geq K_\varepsilon L_0$, applying (\ref{contoff}) and (\ref{lyexoff}), then we have $p_1\geq K_\varepsilon$,
\begin{equation}\label{a}
\begin{aligned}
\|d_zf^{m-k}|_{\widetilde{F}_j(z)}\|&\leq \prod_{\iota=0}^{p_1-1}\|d_{f^{\iota L_0}z}f^{L_0}|_{\widetilde{F}_j(f^{\iota L_0}z)}\|\cdot Q_2 e^\varepsilon\\
&\leq \Big(\prod_{\iota=0}^{p_1-1}\|d_{f^{\iota L_0}(f^kx_0)}f^{L_0}|_{F_j(f^{\iota L_0}(f^kx_0))}\|\Big) e^{p_1\varepsilon}\cdot Q_2 e^\varepsilon \\
&\leq e^{p_1L_0[\lambda_j(\mu)+2\varepsilon]}\cdot e^{p_1L_0\varepsilon}\cdot Q_2  \\
&\leq e^{p_1L_0[\lambda_j(\mu)+3\varepsilon]}\cdot Q_2
\end{aligned}
\end{equation}
and
\begin{equation}\label{b}
\begin{aligned}
m(d_zf^{m-k}|_{\widetilde{G}_j(z)})^{-1}&\leq \prod_{\iota=0}^{p_1-1} m(d_{f^{\iota L_0}z}f^{L_0}|_{\widetilde{G}_j(f^{\iota L_0}z)})^{-1}\cdot Q_1 e^\varepsilon\\
&\leq \prod_{\iota=0}^{p_1-1} m(d_{f^{\iota L_0}x_0}f^{L_0}|_{G_j(f^{\iota L_0}x_0)})^{-1}\cdot e^{p_1\varepsilon}\cdot Q_1 e^\varepsilon  \\
&\leq e^{p_1L_0[-\lambda_{j+1}(\mu)+3\varepsilon]}\cdot Q_1.
\end{aligned}
\end{equation}
If $m-k\geq K_\varepsilon L_0$, by (\ref{contoff}) and (\ref{lyexoff-1}), we obtain
\begin{equation}\label{c}
\begin{aligned}
m(d_zf^{m-k}|_{\widetilde{F}_1(z)}) &\geq \prod_{\iota=0}^{p_1-1} m(d_{f^{\iota L_0}z}f^{L_0}|_{\widetilde{F}_1(f^{\iota L_0}z)}) \cdot Q_2^{-1} e^{-\varepsilon}\\
&\geq \Big(\prod_{\iota=0}^{p_1-1} m( d_{f^{\iota L_0}(f^kx_0)}f^{L_0}|_{F_1(f^{\iota L_0}(f^kx_0))})\Big) e^{-p_1\varepsilon}\cdot Q_2^{-1} e^{-\varepsilon}  \\
&\geq e^{p_1L_0[\lambda_1(\mu)-2\varepsilon]}\cdot e^{-p_1L_0\varepsilon}\cdot Q_2^{-1}  \\
&\geq e^{p_1L_0[\lambda_1(\mu)-3\varepsilon]}\cdot Q_2^{-1}
\end{aligned}
\end{equation}
and
\begin{equation}\label{d}
\begin{aligned}
\|d_zf^{m-k}|_{\widetilde{G}_{l-1}(z)}\| &\leq \prod_{\iota=0}^{p_1-1} \|d_{f^{\iota L_0}z}f^{L_0}|_{\widetilde{G}_{l-1}(f^{\iota L_0}z)}\| \cdot Q_1 e^\varepsilon\\
&\leq \prod_{\iota=0}^{p_1-1} \|d_{f^{\iota L_0}x_0}f^{L_0}|_{G_{l-1}(f^{\iota L_0}x_0)}\| \cdot e^{p_1\varepsilon}\cdot Q_1 e^\varepsilon \\
&\leq e^{p_1L_0[\lambda_{l}(\mu)+3\varepsilon]}\cdot Q_1.
\end{aligned}
\end{equation}
If $m-k<K_\varepsilon L_0$, combining (\ref{contoff}) and the definition of $Q$, then we have $p_1<K_\varepsilon$,
\begin{equation}\label{e}
\begin{aligned}
\|d_zf^{m-k}|_{\widetilde{F}_j(z)}\| &\leq Q\cdot Q_2\cdot e^{p_1L_0\varepsilon},
\end{aligned}
\end{equation}
\begin{equation}\label{f}
\begin{aligned}
m(d_zf^{m-k}|_{\widetilde{G}_j(z)})^{-1}& \leq Q\cdot Q_1\cdot e^{p_1L_0\varepsilon},
\end{aligned}
\end{equation}
\begin{equation}\label{g}
\begin{aligned}
m(d_zf^{m-k}|_{\widetilde{F}_1(z)}) &\geq Q^{-1}\cdot Q_2^{-1}\cdot e^{-p_1L_0\varepsilon}
\end{aligned}
\end{equation}
and
\begin{equation}\label{h}
\begin{aligned}
\|d_zf^{m-k}|_{\widetilde{G}_{l-1}(z)}\|& \leq Q\cdot Q_1\cdot e^{p_1L_0\varepsilon}.
\end{aligned}
\end{equation}
Similarly we obtain that if $k\geq K_\varepsilon L_0$, then $p_2\geq K_\varepsilon$ and
\begin{equation}\label{i}
\begin{aligned}
\|d_{f^{m-k}z}f^k|_{\widetilde{F}_j(f^{m-k}z)}\|&\leq e^{p_2L_0[\lambda_j(\mu)+3\varepsilon]}\cdot Q_2,
\end{aligned}
\end{equation}
\begin{equation}\label{j}
\begin{aligned}
m(d_{f^{m-k}z}f^k|_{\widetilde{G}_j(f^{m-k}z)})^{-1}&\leq e^{p_2L_0[-\lambda_{j+1}(\mu)+3\varepsilon]}\cdot Q_1,
\end{aligned}
\end{equation}
\begin{equation}\label{k}
\begin{aligned}
m(d_{f^{m-k}z}f^k|_{\widetilde{F}_1(f^{m-k}z)}) &\geq e^{p_2L_0[\lambda_1(\mu)-3\varepsilon]}\cdot Q_2^{-1}
\end{aligned}
\end{equation}
and
\begin{equation}\label{l}
\begin{aligned}
\|d_{f^{m-k}z}f^k|_{\widetilde{G}_{l-1}(f^{m-k}z)}\| &\leq e^{p_2L_0[\lambda_{l}(\mu)+3\varepsilon]}\cdot Q_1;
\end{aligned}
\end{equation}
if $k<K_\varepsilon L_0$, then $p_2< K_\varepsilon$,
\begin{equation}\label{m}
\begin{aligned}
\|d_{f^{m-k}z}f^k|_{\widetilde{F}_j(f^{m-k}z)}\|&\leq  Q\cdot Q_2\cdot e^{p_2L_0\varepsilon},
\end{aligned}
\end{equation}
\begin{equation}\label{n}
\begin{aligned}
m(d_{f^{m-k}z}f^k|_{\widetilde{G}_j(f^{m-k}z)})^{-1}&\leq Q\cdot Q_1\cdot e^{p_2L_0\varepsilon},
\end{aligned}
\end{equation}
\begin{equation}\label{o}
\begin{aligned}
m(d_{f^{m-k}z}f^k|_{\widetilde{F}_1(f^{m-k}z)})&\geq  Q^{-1}\cdot Q_2^{-1}\cdot e^{-p_2L_0\varepsilon}
\end{aligned}
\end{equation}
and
\begin{equation}\label{p}
\begin{aligned}
\|d_{f^{m-k}z}f^k|_{\widetilde{G}_{l-1}(f^{m-k}z)}\| &\leq Q\cdot Q_1\cdot e^{p_2L_0\varepsilon}.
\end{aligned}
\end{equation}
Then the proof of the second statement can be divided into three cases:

{\bf Case I:} If $m-k\geq K_\varepsilon L_0$ and $k\geq K_\varepsilon L_0$, then from (\ref{q}), (\ref{a}) and (\ref{i}) we have
\[ \|d_zf^m(v^s)\|\leq \kappa\theta Q_2^2\cdot e^{(p_1+p_2)L_0[\lambda_j(\mu)+3\varepsilon]} \cdot\|v^u\|. \]
By (\ref{r}), (\ref{b}) and (\ref{j}) we obtain
\[ \|v^u\|\leq Q_1^2\cdot e^{(p_1+p_2)L_0[-\lambda_{j+1}(\mu)+3\varepsilon]}\cdot \|d_zf^m(v^u)\|. \]
Thus
\[\|d_zf^m(v^s)\|\leq \kappa\theta Q_1^2 Q_2^2 \cdot e^{(p_1+p_2)L_0[\lambda_j(\mu)-\lambda_{j+1}(\mu)+6\varepsilon]} \cdot \|d_zf^m(v^u)\|.  \]
Therefore
\[\frac{\|d_zf^m(v)\|}{\|v\|}\geq  \frac{1}{1+\kappa^2\theta^2} \Big\{1-\kappa\theta Q_1^2 Q_2^2 e^{(p_1+p_2)L_0[\lambda_j(\mu)-\lambda_{j+1}(\mu)+6\varepsilon]}\Big\} \cdot \frac{\|d_zf^m(v^u)\|}{\|v^u\|}.             \]
It follows from $(\ref{star})$ and $m\geq N_5$ that
\[ m>\max\{\frac{4\log Q_1}{\varepsilon}, \frac{4\log Q_2}{\varepsilon}, \frac{2L_0[\lambda_l(\mu)-3\varepsilon]}{\varepsilon}, \frac{2L_0(\overline{\vartheta}-6\varepsilon)}{\varepsilon}\}.\]
Then we conclude
\[\frac{\|d_zf^m(v)\|}{\|v\|}\geq  \frac{1}{1+\kappa^2\theta^2} \Big\{1-\kappa\theta e^{m[\lambda_j(\mu)-\lambda_{j+1}(\mu)+8\varepsilon]}\Big\} \cdot e^{m[\lambda_{j+1}(\mu)-5\varepsilon]}.\]
As $m$ is large enough, it yields that
\[\frac{\|d_zf^m(v)\|}{\|v\|}\geq e^{m[\lambda_{j+1}(\mu)-6\varepsilon]}.\]

{\bf Case II:} If $m-k\geq K_\varepsilon L_0$ and $k<K_\varepsilon L_0$, by $(\ref{q})$, $(\ref{a})$ and $(\ref{m})$, we have
\[ \|d_zf^m(v^s)\|\leq \kappa\theta Q Q_2^2\cdot e^{p_1L_0[\lambda_j(\mu)+3\varepsilon]+p_2L_0\varepsilon} \cdot\|v^u\|. \]
By $(\ref{r})$, $(\ref{b})$ and $(\ref{n})$, we conclude that
\[ \|v^u\|\leq Q Q_1^2\cdot e^{p_1L_0[-\lambda_{j+1}(\mu)+3\varepsilon]+p_2L_0\varepsilon}\cdot \|d_zf^m(v^u)\|.
\]
Thus
\[\|d_zf^m(v^s)\|\leq \kappa\theta Q^2 Q_1^2 Q_2^2 \cdot e^{p_1L_0[\lambda_j(\mu)-\lambda_{j+1}(\mu)+6\varepsilon]+m\varepsilon} \cdot \|d_zf^m(v^u)\|.  \]
Therefore
\[\frac{\|d_zf^m(v)\|}{\|v\|}\geq  \frac{1}{1+\kappa^2\theta^2} \Big\{1-\kappa\theta Q^2 Q_1^2 Q_2^2 e^{p_1L_0[\lambda_j(\mu)-\lambda_{j+1}(\mu)+6\varepsilon]+m\varepsilon}\Big\} \cdot \frac{\|d_zf^m(v^u)\|}{\|v^u\|}.             \]
Since
$m\geq N_5>\max\{\frac{4\log Q}{\varepsilon}, \frac{4\log Q_1}{\varepsilon}, \frac{4\log Q_2}{\varepsilon}, \frac{2(K_\varepsilon+1)L_0(\overline{\vartheta}-6\varepsilon)}{\varepsilon}, \frac{4(K_\varepsilon+1)L_0[\lambda_l(\mu)-3\varepsilon]}{\varepsilon}\}$, we obtain
\[\frac{\|d_zf^m(v)\|}{\|v\|}\geq  \frac{1}{1+\kappa^2\theta^2} \Big\{1-\kappa\theta e^{m[\lambda_j(\mu)-\lambda_{j+1}(\mu)+9\varepsilon]}\Big\} \cdot e^{m[\lambda_{j+1}(\mu)-5\varepsilon]}.\]
It yields that
\[\frac{\|d_zf^m(v)\|}{\|v\|}\geq e^{m[\lambda_{j+1}(\mu)-6\varepsilon]}.\]

{\bf Case III:} If $m-k<K_\varepsilon L_0$ and $k\geq K_\varepsilon L_0$,
it is parallel to Case II. It follows from $(\ref{q})$, $(\ref{r})$, $(\ref{e})$ $(\ref{f})$, $(\ref{i})$ and $(\ref{j})$ that
\[  \frac{\|d_zf^m(v)\|}{\|v\|}\geq e^{m[\lambda_{j+1}(\mu)-6\varepsilon]}.\]

Finally, we prove the third statement of the corollary. For any $z\in\Lambda$, $0\neq v\in C_\theta^{G_{l-1}}(z)$ with $v=v^s+v^u$ with $v^s\in\widetilde{F}_{l-1}(z)$, $v^u\in\widetilde{G}_{l-1}(z)$ and $\|v^s\|\leq\kappa\theta\|v^u\|$.
If $m-k\geq K_\varepsilon L_0$ and $k\geq K_\varepsilon L_0$, by $(\ref{q})$, $(\ref{a})$, $(\ref{i})$ and
$(\ref{d})$, $(\ref{l})$, we have
\begin{equation*}
\begin{aligned}
\|d_zf^m(v^s)\|&\leq Q_2^2\cdot e^{(p_1+p_2)L_0[\lambda_{l-1}(\mu)+3\varepsilon]}\cdot \|v^s\|,\\
\|d_zf^m(v^u)\|&=\|d_{f^{m-k}z}f^k\circ d_zf^{m-k}(v^u)\|\\
&\leq \|d_{f^{m-k}z}f^k|_{\widetilde{G}_{l-1}(f^{m-k}z)}\| \cdot \|d_zf^{m-k}|_{\widetilde{G}_{l-1}(z)}\| \cdot \|v^u\|\\
&\leq e^{(p_1+p_2)L_0[\lambda_l(\mu)+3\varepsilon]} \cdot Q^2_1\cdot \|v^u\|.
\end{aligned}
\end{equation*}
Since $m\geq \max\{\frac{4\log Q_1}{\varepsilon}, \frac{4\log Q_2}{\varepsilon}, \frac{4(K_\varepsilon+1)L_0[\lambda_l(\mu)-3\varepsilon]}{\varepsilon}\}$, then we conclude
\begin{equation*}
\begin{aligned}
\|d_zf^m(v^s)\|&\leq e^{m[\lambda_l(\mu)+4\varepsilon]}\cdot \|v^s\|,\\
\|d_zf^m(v^u)\|&\leq e^{m[\lambda_l(\mu)+4\varepsilon]}\cdot \|v^u\|.
\end{aligned}
\end{equation*}
It yields that
\begin{equation*}
\begin{aligned}
\frac{\|d_zf^m(v)\|}{\|v\|} &\leq \frac{\|d_zf^m(v^u)\|+\|d_zf^m(v^s)\|}{\|v\|}\\
&\leq\frac{\|v^u\|}{\|v\|}\cdot \frac{\|d_zf^m(v^u)\|+\|d_zf^m(v^s)\|}{\|v^u\|}\\
&\leq \frac{1}{1-\kappa\theta} \frac{\|d_zf^m(v^u)\|+\|d_zf^m(v^s)\|}{\|v^u\|}\\
&\leq \frac{1}{1-\kappa\theta} (1+\kappa\theta)e^{m[\lambda_l(\mu)+4\varepsilon]}\\
&\leq e^{m[\lambda_l(\mu)+6\varepsilon]}
\end{aligned}
\end{equation*}
the last inequality is because that $m$ is large enough.
If $m-k\geq K_\varepsilon L_0$ and $k<K_\varepsilon L_0$, by
$(\ref{q})$, $(\ref{a})$, $(\ref{m})$ and
$(\ref{d})$, $(\ref{p})$, $(\ref{star})$, $m\geq N_5$,
then we conclude that
\begin{equation*}
\begin{aligned}
\|d_zf^m(v^s)\|&\leq Q Q_2^2\cdot e^{p_1L_0[\lambda_{l-1}(\mu)+3\varepsilon]+p_2L_0\varepsilon}\cdot \|v^s\|\\
&\leq e^{m[\lambda_l(\mu)+4\varepsilon]}\cdot \|v^s\|,\\
\|d_zf^m(v^u)\|&\leq Q_1\cdot e^{p_1L_0[\lambda_l(\mu)+3\varepsilon]} \cdot QQ_1 e^{p_2L_0\varepsilon} \cdot \|v^u\|\\
&\leq e^{m[\lambda_l(\mu)+4\varepsilon]}\cdot \|v^u\|.
\end{aligned}
\end{equation*}
Therefore
\[ \frac{\|d_zf^m(v)\|}{\|v\|}\leq e^{m[\lambda_l(\mu)+6\varepsilon]}. \]
If $m-k<K_\varepsilon L_0$ and $k\geq K_\varepsilon L_0$, then we can also prove
\[ \frac{\|d_zf^m(v)\|}{\|v\|}\leq e^{m[\lambda_l(\mu)+6\varepsilon]} \]
by the same way as in the case of $m-k\geq K_\varepsilon L_0$ and $k<K_\varepsilon L_0$. Thus we proved for any $z\in\Lambda$, $0\neq v\in C_\theta^{G_{l-1}}(z)$
\[ \|d_zf^m(v)\| \leq e^{m[\lambda_l(\mu)+6\varepsilon]} \|v\|. \]

By considering $f^{-1}$, we can similarly show that for any $n\in\mathbb{Z}$, $0\neq w\in C^{F_j}_{\theta}(f^{nm}z)$, $j=1,2,\cdots,l-1$,
\[ \|d_{f^{nm}z}f^{-m}(w)\|\geq e^{[-\lambda_{j}(\mu)-6\varepsilon]m}\|w\|,\]
and for any nonzero vector $w\in C^{F_1}_{\theta}(f^{nm}z)$,
\[ \|d_{f^{nm}z}f^{-m}(w)\|\leq e^{[-\lambda_{1}(\mu)+6\varepsilon]m}\|w\|.\]
Therefore we complete the proof of the corollary.
\end{proof}

In the paragraphs to follow, we will construct a dominated splitting corresponding to Oseledec subspace on $\Lambda$.
For any $z\in\Lambda$, let
\begin{equation*}
\begin{aligned}
\widetilde{E}_1(z) &= \widetilde{F}_1(z),\\
 \widetilde{E}_j(z) &= \widetilde{F}_j(z) \cap \widetilde{G}_{j-1}(z),\ j=2,3,\cdots,l-1,\\
 \widetilde{E}_{l}(z) &= \widetilde{G}_{l-1}(z).
 \end{aligned}
 \end{equation*}
By Lemma \ref{ds}, we conclude that the splitting
\[ T_\Lambda M=\widetilde{E}_1\oplus \widetilde{E}_2\oplus\cdots\oplus \widetilde{E}_l    \]
is dominated. Therefore the splitting
\[ T_z M=\widetilde{E}_1(z)\oplus \widetilde{E}_2(z)\oplus\cdots\oplus \widetilde{E}_l(z)    \]
varies continuously with the point $z\in\Lambda$.
By Corollary \ref{corofle} and (\ref{biaohao}), we have
\begin{equation}\label{3}
  e^{m[\lambda_j(\mu)-6\varepsilon]}\|w\| \leq \|d_zf^m(w)\| \leq e^{m[\lambda_j(\mu)+6\varepsilon]}\|w\|
\end{equation}
for any nonzero vector $w\in \widetilde{E}_j(z)$, $j=1,2,\cdots,l$ and $z\in\Lambda$.


(i)
By Lemma \ref{invariantofcone} and Theorem \ref{hyperbolicity},
$\Lambda^*$ is a hyperbolic set with respect to $f^m$. Since $f^m|_{\Lambda^*}$ is topologically conjugate to a full two-side shift in the symbolic space with $Card(\Delta_m\cap P)$ symbols,
it is topologically mixing with respect to $f^m$. We proved (i) of the main Theorem.

(ii)Since
$ h_{top}(f|_\Lambda)= \frac{1}{m}h_{top}(f^m|\Lambda^*)= \frac{1}{m}\log Card(\Delta_m\cap P)$,
by (\ref{uplowbounded}), it yields that
\[ h_{top}(f|_\Lambda)\geq-\frac{1}{m}\log t+\frac{N_5}{m}\cdot[h_\mu(f)-\varepsilon]. \]
Combining $N_5>\frac{4}{\varepsilon}\log t$ and $N_5\leq m< N_5(1+\varepsilon)$, we have
\begin{eqnarray*}
\begin{aligned}
h_{top}(f|_\Lambda)&\geq -\frac14\varepsilon+\frac{1}{1+\varepsilon}\cdot[h_\mu(f)-\varepsilon]\geq h_\mu(f)-[h_\mu(f)+2]\varepsilon.
\end{aligned}
\end{eqnarray*}
Applying (\ref{uplowbounded}), we conclude that
\[ h_{top}(f|_\Lambda)\leq \frac{N_5}{m}\cdot[h_\mu(f)+\varepsilon]\leq h_\mu(f)+\varepsilon. \]
This shows (ii) of the theorem.

(iii) Since $\rho<\frac{\varepsilon}{2}$, combining  Claim \ref{hyperinne}, we prove (iii) of the main theorem.

(iv)For any $f$-invariant probability measure $\nu$ supported on $\Lambda$, suppose $\nu$ is ergodic,
\begin{eqnarray*}
\begin{aligned}
d(\nu,\mu)&= \sum_{j=1}^\infty 2^{-j}\frac{1}{2\|\varphi_j\|_\infty}\cdot\Big|\int\varphi_jd\mu-\int\varphi_j d\nu\Big|\\
&= \sum_{j=1}^J 2^{-j}\frac{1}{2\|\varphi_j\|_\infty}\cdot\Big|\int\varphi_jd\mu-\int\varphi_j d\nu\Big|\\
&\ \ \ \ +\sum_{j=J+1}^\infty 2^{-j}\frac{1}{2\|\varphi_j\|_\infty}\cdot\Big|\int\varphi_jd\mu-\int\varphi_j d\nu\Big|\\
&\leq \sum_{j=1}^J 2^{-j}\frac{1}{2\|\varphi_j\|_\infty}\cdot\Big|\int\varphi_jd\mu-\int\varphi_j d\nu\Big|+\frac{1}{2^{J-1}}\\
&\leq \sum_{j=1}^J 2^{-j}\frac{1}{2\|\varphi_j\|_\infty}\cdot\Big|\int\varphi_jd\mu-\int\varphi_j d\nu\Big|+\frac{\varepsilon}{4}.
\end{aligned}
\end{eqnarray*}
We claim that $|\int\varphi_jd\mu-\int\varphi_jd\nu|\leq\frac{3\|\varphi_j\|_\infty}{4}\varepsilon \text{ for } j=1,2,\cdots,J.$
In fact, choosing $y\in\Lambda^*$ and $s\in\mathbb{N}$ large enough such that
$$\Big|\frac{1}{ms}\sum_{k=0}^{ms-1}\varphi_j(f^ky)-\int\varphi_jd\nu\Big|\leq\frac{\|\varphi_j\|_\infty}{4}\varepsilon$$
for $j=1,2,\cdots,J$. Then there exist $x_0,x_1,\cdots,x_{s-1}\in \Delta_m\cap P$ such that
$d(f^{km+t}y,f^tx_k)\leq\rho$
for $0\leq k\leq s-1$ and $0\leq t\leq m-1$. By (\ref{continuous}) and the construction of $\Lambda_{H,\delta,m}$,
\begin{eqnarray*}
\begin{aligned}
&\ \ \  \  \Big|\int\varphi_jd\mu-\int\varphi_jd\nu\Big|\\
&\leq \Big|\frac{1}{s}s\int\varphi_jd\mu-\frac{1}{s}\cdot\Big(\frac{1}{m}\sum_{k=0}^{m-1}\varphi_j(f^kx_0)+\frac{1}{m}\sum_{k=0}^{m-1}\varphi_j(f^kx_1)+\cdots+\frac{1}{m}\sum_{k=0}^{m-1}\varphi_j(f^kx_{s-1})\Big)\Big|\\
& +\Big|\frac{1}{ms}\Big(\sum_{k=0}^{m-1}\varphi_j(f^kx_0)+\cdots+\sum_{k=0}^{m-1}\varphi_j(f^kx_{s-1})\Big)-\frac{1}{ms}\sum_{k=0}^{ms-1}\varphi_j(f^ky)\Big|\\
& +\Big|\frac{1}{ms}\sum_{k=0}^{ms-1}\varphi_j(f^ky)-\int\varphi_jd\nu\Big|\\
&\leq \frac{\|\varphi_j\|_\infty}{4}\varepsilon+\frac{\|\varphi_j\|_\infty}{4}\varepsilon+\frac{\|\varphi_j\|_\infty}{4}\varepsilon=\frac{3\|\varphi_j\|_\infty}{4}\varepsilon.
\end{aligned}
\end{eqnarray*}
Therefore the proof of the claim is complete.
It implies that $d(\mu,\nu)\leq\frac{3}{4}\varepsilon+\frac{\varepsilon}{4}=\varepsilon$ for any ergodic measure $\nu$. If $\nu$ is not ergodic, the ergodic decompositional theorem tells us that $\nu$-almost every ergodic component is supported on $\Lambda$, thus
\begin{eqnarray*}
\begin{aligned}
d(\mu,\nu)&= \sum_{j=1}^\infty\frac{\Big|\int\varphi_jd\mu-\int\varphi_jd\nu\Big|}{2^j\cdot2\cdot\|\varphi_j\|_\infty}= \sum_{j=1}^\infty\frac{\Big|\int\big(\int\varphi_jd\mu-\int\varphi_jd\nu_x\big)d\nu(x)\Big|}{2^j\cdot2\cdot\|\varphi_j\|_\infty}\\
&\leq \int d(\mu,\nu_x)d\nu(x)\leq \varepsilon.
\end{aligned}
\end{eqnarray*}
This proves (iv) of the theorem.

(v)Since $\Lambda^*$ is a hyperbolic set with respective to $f^m$ and $\Lambda=\Lambda^*\cup f(\Lambda^*)\cup\cdots\cup f^{m-1}(\Lambda^*)$, combining with (\ref{3}), we obtain conclusion (v) of Theorem \ref{maintheorem}.
\end{proof}



\bibliographystyle{alpha}
\bibliography{bib}

\end{document}